\documentclass[12pt,twoside]{article}
\usepackage[latin1]{inputenc}
\usepackage{times}
\usepackage{epsfig}
\usepackage{amsfonts}
\usepackage{float}
\usepackage{amsmath}
\usepackage{amsthm}
\usepackage{amssymb}
\usepackage{latexsym}
\usepackage{graphicx, color}
\usepackage{caption}
\usepackage{subcaption}
\usepackage{subfig}
\usepackage{booktabs}
\usepackage{rotating}
\usepackage{xr}

\relax 
\allowdisplaybreaks[3]

\newtheorem{definition}{{\sc Definition}\sc}[section]
\newcommand{\bdefi}{\begin{definition}}
\newcommand{\edefi}{\end{definition}}

\newtheorem{appropr}[definition]{{\sc Approximation Procedure}\sc}
\newcommand{\bappr}{\begin{appropr}}
\newcommand{\eappr}{\end{appropr}}

\newtheorem{bedi}[definition]{{\sc Condition}\sc}
\newcommand{\bbd}{\begin{bedi}}
\newcommand{\ebd}{\end{bedi}}

\newtheorem{bedin}[definition]{{\sc Conditions}\sc}
\newcommand{\bbdn}{\begin{bedin}}
\newcommand{\ebdn}{\end{bedin}}

\newtheorem{corollary}[definition]{{\sc Corollary}\sc}
\newcommand{\bco}{\begin{corollary}}
\newcommand{\eco}{\end{corollary}}

\newtheorem{lemma}[definition]{{\sc Lemma}\sc}
\newcommand{\blem}{\begin{lemma}}
\newcommand{\elem}{\end{lemma}}

\newtheorem{proposition}[definition]{{\sc Proposition}\sc}
\newcommand{\bpro}{\begin{proposition}}
\newcommand{\epro}{\end{proposition}}

\newtheorem{satz}[definition]{{\sc Theorem}\sc}
\newcommand{\bsa}{\begin{satz}}
\newcommand{\esa}{\end{satz}}

\newtheorem{theorem}[definition]{{\sc Theorem}\sc}
\newcommand{\bth}{\begin{theorem}}

\newtheorem{assumption}[definition]{{\sc Assumption}\sc}
\newcommand{\bas}{\begin{assumption}}
\newcommand{\eas}{\end{assumption}}

\newtheorem{assumptions}[definition]{{\sc Assumptions}\sc}
\newcommand{\bass}{\begin{assumptions}}
\newcommand{\eass}{\end{assumptions}}

\newtheorem{abb}{{\sc Figure}\sc}
\newcommand{\babb}{\begin{abb}}
\newcommand{\eabb}{\end{abb}}

\newenvironment{remark}{\begin{rmk}\sl}{\end{rmk}}
\newtheorem{rmk}{{\sc Remark}\sc}[section]
\newcommand{\brem}{\begin{remark}}
\newcommand{\erem}{\end{remark}}

\newenvironment{remarks}{\begin{rmks}\sl}{\end{rmks}}
\newtheorem{rmks}{{\sc Remarks}\sc}[section]
\newcommand{\brems}{\begin{remarks}}
\newcommand{\erems}{\end{remarks}}

\newenvironment{example}{\begin{exmp}\rm}{\end{exmp}}
\newtheorem{exmp}{{\sc Example}\sc}[section]
\newcommand{\bbsp}{\begin{example}}
\newcommand{\ebsp}{\end{example}}
\newcommand{\bexa}{\begin{example}}
\newcommand{\eexa}{\end{example}}

\newtheorem{model}{{\sc Model}\sc}[section]
\newcommand{\bmdl}{\begin{model}}
\newcommand{\emdl}{\end{model}}

\newtheorem{scheme}{{\sc Scheme}\sc}[section]
\newcommand{\bscm}{\begin{scheme}}
\newcommand{\escm}{\end{scheme}}

\newenvironment{tabelle}{\begin{tabl}\rm}{\end{tabl}}
\newtheorem{tabl}{{\bf Table}}
\newcommand{\btab}{\begin{tabelle}}
\newcommand{\etab}{\end{tabelle}}

\newenvironment{exercise}{\begin{exc}\sl}{\end{exc}}
\newtheorem{exc}{Exercise}[section]
\newcommand{\bexe}{\begin{exercise}}
\newcommand{\eexe}{\end{exercise}}

\relax

\newcommand{\bay}{\begin{array}}
\newcommand{\eay}{\end{array}}

\newcommand{\bqa}{\begin{eqnarray*}}
\newcommand{\eqa}{\end{eqnarray*}}

\newcommand{\bee}{\begin{eqnarray*}}
\newcommand{\eee}{\end{eqnarray*}}

\newcommand{\bea}{\begin{eqnarray*}}
\newcommand{\eea}{\end{eqnarray*}}

\newcommand{\bqan}{\begin{eqnarray}}
\newcommand{\eqan}{\end{eqnarray}}

\newcommand{\be}{\begin{eqnarray}}
\newcommand{\ee}{\end{eqnarray}}

\newcommand{\bit}{\begin{itemize}}
\newcommand{\eit}{\end{itemize}}

\newcommand{\ben}{\begin{enumerate}}
\newcommand{\een}{\end{enumerate}}

\newcommand{\beq}{\begin{equation}}
\newcommand{\eeq}{\end{equation}}

\newcommand{\bdes}{\begin{description}}
\newcommand{\edes}{\end{description}}

\newcommand{\btb}{\begin{tabular}}
\newcommand{\etb}{\end{tabular}}

\newcommand{\bcen}{\begin{center}}
\newcommand{\ecen}{\end{center}}

\newcommand{\bmp}{\begin{minipage}}
\newcommand{\emp}{\end{minipage}}

\newtheorem{thm}{Theorem}
\newtheorem{rem}{Remark}
\newtheorem{cor}{Corollary}

\newcommand{\rnc}{\renewcommand}
\newcommand{\nc}{\newcommand}
\newcommand{\mrm}{\mathrm}
\renewcommand{\b}{\textbf}
\nc{\leb}{\lambda \! \! \lambda}
\nc{\mb}{\mathbb}
\nc{\mac}{\mathcal}
\nc{\E}{\mb{E}}
\nc{\N}{\mb{N}}
\nc{\R}{\mb{R}}
\nc{\Q}{\mb{Q}}
\rnc{\P}{\mrm P}
\rnc{\d}{\mrm d}
\nc{\C}{\mac{C}}
\nc{\D}{\mac{D}}
\nc{\B}{\mac{B}}
\nc{\oPo}{\stackrel{p}{\longrightarrow}}
\nc{\oWo}{\stackrel{w}{\longrightarrow}}
\nc{\oDo}{\stackrel{d}{\longrightarrow}}
\nc{\gDg}{\stackrel{d}{=}}

\numberwithin{equation}{section}

\begin{document}

\title{\Large \bf Approximative Tests for the Equality of\\ Two Cumulative Incidence Functions of a Competing Risk}
\author{Dennis Dobler$^*$ and Markus Pauly$^*$ 
}
\maketitle

\noindent${}^*$ {University of Ulm, Institute of Statistics, Germany, \\
\mbox{ }\hspace{1 ex}
email: dennis.dobler@uni-ulm.de  (corresponding author) \\
\mbox{ }\hspace{1 ex}
email: markus.pauly@uni-ulm.de}


\begin{abstract}
In the context of a competing risks set-up we discuss different inference procedures for testing equality of two cumulative incidence functions, where the data may be subject to
independent right-censoring or left-truncation.
To this end, we compare two-sample Kolmogorov-Smirnov- and Cram\'er-von Mises-type test statistics. 
Since, in general, their corresponding asymptotic limit distributions depend on unknown quantities,
we utilize wild bootstrap resampling as well as approximation techniques to construct adequate test decisions. 
Here the latter procedures are motivated from testing procedures for heteroscedastic factorial designs 
but have not yet been proposed in the survival context. 
A simulation study shows the performance of all considered tests under various settings {\color{black} and finally a real data example 
about bloodstream infection during neutropenia is used to illustrate their application.}
\end{abstract}


 \newpage

%
\section{Introduction}
%

We study non-parametric inference procedures for testing equality of cumulative incidence functions (CIFs) of a competing risk in an independent two-sample set-up. 
Typically, time-simultaneous inference for a CIF is based on the Aalen-Johansen estimator (AJE); see Aalen and Johansen (1978). \nocite{AalenJoh78}
However, due to its complicated limit distribution, additional techniques are needed to gain AJE-based inference methods. 
For example, when constructing simultaneous confidence bands for a CIF, this is often attacked by means of Lin's resampling method; see Lin et al. (1993), Lin (1997) or the monograph of Martinussen and Scheike (2006). 
\nocite{lin93}\nocite{lin97}\nocite{martinussen06}

Recently, it has been seen that his technique is a special example of
the general wild bootstrap; see Cai et al. (2010), Elgmati et al. (2010) or Beyersmann et al. (2013).
\nocite{cai10}\nocite{elgmati10}\nocite{beyersmann13}
Moreover, weak convergence of the wild bootstrap and of other weighted as well as data-dependent bootstrap versions of the AJE have been rigorously studied in 
Beyersmann et al. (2013) as well as in Dobler and Pauly (2014) and Dobler et al. (2015).\nocite{dobler15b}\nocite{dobler14}
As pointed out in Bajorunaite and Klein (2007, 2008)\nocite{bajorunaite07}\nocite{bajorunaite08}, Sankaran et al. (2010),\nocite{sankaran10}
and Dobler and Pauly (2014), Lin's resampling scheme as well as the more general wild bootstrap can also be applied for two-sample problems concerning CIFs. 
In particular, the aforementioned papers discuss different wild bootstrap-based tests 
for ordered and/or equal CIFs. 
However, especially the simulation studies in Bajorunaite and Klein (2007) show that, e.g., Kolomogorov-Smirnov-type tests based on Lin's wild bootstrap may be extremely liberal for small sample sizes. 

To overcome this problem, we study additional testing procedures. In particular, we utilize 
several approximation techniques which have been independently developed for constructing conservative tests for heteroscedastic 
factorial designs; see e.g., the generalized Welch-James test (Johansen, 1980), the ANOVA-type statistic suggested by Brunner et al.
(1997), or the approximate degree of freedom test by Zhang (2012). \nocite{Johansen80, zhang12, BDM97}
There the main idea is to approximate the limit distribution of underlying quadratic forms (which is mostly of weighted $\chi_1^2$-form) by adequate transformations of 
$\chi^2_f$-distributions with estimated degrees of freedom. For example, the famous Box approximation, see Box (1954)\nocite{box54}, is obtained by matching expectation and variance of the statistic 
with a scaled $g\chi_f^2$-distribution. Moreover, additionally matching its skewness, the Pearson approximation is obtained, see Pearson (1959) or Pauly et al. (2013). \nocite{Pearson, PEB13}
In the current paper we apply this approach to two-sample Cram\'er-von Mises-type statistics in AJEs. 
We like to point out that all procedures are motivated from competing risks designs with independent left-truncation and right-censoring but can also be constructed  
for more general counting processes satisfying the multiplicative intensity model.

The paper is organized as follows. The statistical model, the considered estimators and their large sample behaviour are introduced in Section~2. 
In Section 3 we present different test statistics as functionals of these estimators, where we distinguish between bootstrap-based and approximative tests.
Their finite sample properties are investigated in a simulation study given in Section 4. 
{\color{black}The developed theory is then applied to a data-set from the ONKO-KISS study about bloodstream infection (BSI) during neutropenia from Dettenkofer et al. 
(2005)~\nocite{dettenkofer05} in Section~\ref{sec: data_ex}, supplementing existing analyses (see, e.g., Beyersmann et al., 2007, 
\nocite{beyersmann2007competing} or Meyer et al., 2007) \nocite{meyer2007risk} with respect to significance testing.}
Occurrence of a BSI during neutropenia, end of neutropenia and death without a preceding BSI induce a competing risks situation
where observation of the eventual outcome may be right-censored.
Finally, we give some concluding remarks in Section 5. All proofs
and further simulation results are deferred to the Appendix.

%
\section{Notation, Model and Estimators}\label{sec: model}
%

Let $X = (X(t))_{t \geq 0}$ be a right-continuous stochastic process with left-hand limits 
and values in a finite state space, $\{0, 1, \dots, m\}, m \geq 2$.
$X$ is called a competing risks process with $m$ competing risks and initial state 0
if $P(X(0) = 0) = 1$ and
if, for all $s \leq t$, the transition probabilities are given as 
$P(X(t) = j \; | \; X(s) = j) = 1, 1 \leq j \leq m$.
That is, each of the states $1,\dots,m$ is absorbing,
in which case $X$ is simply a time-(in)homogeneous Markov process.
From a medical point of view, $X$ may be interpreted as the health status over time of a diseased individual 
who can experience one out of several causes of death.
For ease of notation, we let $X$ henceforth be a competing risks process with $m=2$ absorbing states.
The case of a general number of risks can be dealt with in the same manner.

The event time of $X$ defined as $T = \inf\{ t > 0 : X(t) \neq 0 \}$
is supposedly finite with probability 1.
Therefore, $X(T) \in \{1,2\}$ and $X(T-) = 0$ where the minus indicates the left-hand limit.
Modeling of the specific risks is done via the cause-specific hazard intensities
\begin{align*}
  \alpha_j(t) = \lim_{\delta \downarrow 0} \frac{1}{\delta} P(T \in [t, t + \delta), X(T) = j \; | \; T \geq t), 
  \quad j=1,2,
\end{align*}
which are assumed to exist.
Moreover, $\tau = \sup \{ t \geq 0 : \int_0^t (\alpha_1 + \alpha_2) (s) \d s < \infty \} \in [0, \infty]$
is the endpoint of any possible observation.
With these definitions, we call
\begin{align}
  F_j(t) = P(T \leq t, X(T) = j) = \int_0^t P(T > s -) \alpha_j(s) \d s, \quad j=1,2,
\end{align}
the cumulative incidence functions (CIFs) for causes $j=1,2$
which are zero at time zero, continuous and non-decreasing. For future abbreviations, we also introduce $S_j(t) = 1 - F_j(t)$ 
as the probability not to die of cause $j=1,2$ until time $t$. Some authors also refer to CIFs as sub-distribution functions; 
see, e.g., Gray (1988) \nocite{Gray_1988} or Beyersmann et al. (2012)\nocite{beyersmann12a} for a textbook giving the preceding definitions. For the modeling of CIFs in related (e.g., regression) problems 
we refer to the review papers by Zhang et al. (2008) \nocite{zhang08} and Latouche (2010). \nocite{Latouche10}

Now consider $n$ independent copies of $X$ which may be interpreted as observing $n$ individuals under study. 
Since these processes are not always fully observable, the following counting processes are a necessity for stating proper estimators for $F_j$:
\begin{align*}
 Y_i(t) & = \b 1 \{ \text{ subject i is observed to be in state } 0  \text{ at time } t- \} \\
 N_{j;i}(t) & = \b 1 \{  \text{ subject i has an observed } (0 \rightarrow j) \text{-transition in } [0,t] \},
\end{align*}
$j=1,2, i =1,\dots,n,$ where $\bf 1\{ \cdot \}$ denotes the indicator function.
Hence, let $Y = \sum_{i = 1}^n Y_i$ be the number at risk process
and let the counting process $N_j = \sum_{i = 1}^n N_{j;i}$ count the total number of observed $(0 \rightarrow j)$-transitions.
Further, we suppose that the so-called multiplicative intensity model holds, that is,
$Y \alpha_j$ is the intensity process of $N_j$,
so that 
\begin{align*}
  M_j(t) = \sum_{i=1}^n M_{j;i}(t) = \sum_{i=1}^n \Big( N_{j;i}(t) - \int_0^t Y_i(s) \alpha_j(s) \d s \Big) 
    = N_j(t) - \int_0^t Y(s) \alpha_j(s) \d s
\end{align*}
are local martingales for $j=1,2$. 
For a specification of the associated filtration, we refer to Andersen et al. (1993)\nocite{abgk93}.
Therein, it is also pointed out that, amongst others, the case of left-truncated and right-censored observations satisfies the required multiplicative intensity model;
see Chapter~III and IV in this monograph for these and other models for incomplete data.

In the present context of competing risks, the Aalen-Johansen estimator for the transition probability matrix of Markov processes collapses to 
an estimator for CIFs given as
\begin{align*}
 \hat F_j(t) = \int_0^t \hat P(T > s-) \frac{\d N_j(s)}{Y(s)} ,
\end{align*}
where $\hat P(T > s)$ denotes the Kaplan-Meier estimator for the probability of surviving until time $s$ and the integrand is set to be zero in case $Y(s)=0$. 
Under the assumption that there exists a function $y: [0,t] \rightarrow [0,1]$
satisfying the convergence in probability
\begin{align}
\label{eq:at_risk_conv}
 \sup_{s \in [0,t]} \Big| \frac{Y(s)}{n} - y(s) \Big| \oPo 0,
\end{align}
where $\inf_{s \in [0,t]} y(s) > 0$,
it is seen that the Aalen-Johansen estimator is consistent as well as asymptotically Gaussian.
That is, even weak convergence on the Skorohod space $\D[0,t]$ holds true; 
see, e.g., Section IV.4 in Andersen et al. (1993) or Beyersmann et al. (2013).\nocite{beyersmann13} 
For completeness, we summarize this result.
%
\begin{thm}[Aalen and Johansen, 1978\nocite{AalenJoh78}]
 Let $t < \tau$ and suppose~\eqref{eq:at_risk_conv} holds. 
 Then, as $n \rightarrow \infty$, convergence in distribution
 \begin{align*}
  W_n = \sqrt{n} ( \hat F_1 - F_1 ) \oDo U
 \end{align*}
 holds on the Skorohod space $\D[0,t]$
 where $U$ is a time-continuous, zero-mean Gaussian process with covariance function
 \begin{align}
  \label{eq: covU}\zeta_U(s_1,s_2) = & \int_0^{s_1 \wedge s_2} \{ S_2(u) - F_1(s_1) \}\{ S_2(u) - F_1(s_2) \} \frac{ \alpha_1(u) }{y(u)} \d u \\
    & \quad + \int_0^{s_1 \wedge s_2} \{ F_1(u) - F_1(s_1) \}\{ F_1(u) - F_1(s_2) \} \frac{ \alpha_2(u) }{y(u)} \d u .\nonumber
 \end{align}
\end{thm}
%
Note, that \eqref{eq:at_risk_conv} holds e.g., in case of independent right-censoring and left-truncation; see Examples~IV.1.7. and 1.8. in Andersen et al. (1993).

Since we are interested in two-sample comparisons of CIFs,
we introduce each of the above quantities sample-specifically 
and denote them with a superscript ${ }^{(k)}, k=1,2$.
Moreover, we denote by $n_k$ the sample size of group $k=1,2$
and let $n = n_1 + n_2$ be the total sample size.
Henceforth it is supposed that $\frac{n_1}{n} \rightarrow p \in (0,1)$ holds as $\min(n_1,n_2) \rightarrow \infty$.
Fix a compact interval $I \subset [0,\tau)$, where $\tau := \tau^{(1)} \wedge \tau^{(2)}$.
We are now interested in testing the null hypothesis
\begin{align}
\label{eq:hypotheses}
 H_= : \{ F_1^{(1)} = F_1^{(2)} \text{ on } I \} 
 \; \text{ versus } \;
 H_{\neq} : \{ F_1^{(1)} \neq F_1^{(2)} \text{ on a set } A \subset I \text{ with } \lambda \! \! \lambda(A) > 0 \},
\end{align}
where $\lambda \! \! \lambda$ denotes Lebesgue measure.
An immediate consequence of the above result is the following theorem for comparing sample-specific CIFs:
%
\begin{thm}
\label{thm:sqrt_n_conv_cifs}
 {\color{black}Suppose~\eqref{eq:at_risk_conv} holds for both samples with $[0,t]$ replaced by $I$.} 
 Then, under $H_=$,
 \begin{align*}
  W_{n_1,n_2} = \sqrt{\frac{n_1 n_2}{n}} ( \hat F_1^{(1)} - \hat F_1^{(2)} ) \oDo V
 \end{align*}
 holds on the Skorohod space $\D ( I )$
 where $V$ is a time-continuous, zero-mean Gaussian process with covariance function
 \begin{align}
 \label{eq:zeta_V}
  \zeta_V(s_1,s_2) = &  (1-p) \zeta_U^{(1)}(s_1,s_2) + p \zeta_U^{(2)}(s_1,s_2) .
 \end{align}
Here $\zeta_U^{(k)}, k=1,2,$ is given by \eqref{eq: covU} with superscripts ${ }^{(k)}$ at all quantities in the integrand.
\end{thm}
%
%
%
In the subsequent section it is shown that continuous functionals of $W_{n_1,n_2}$ 
can be used as test statistics for testing the equality of CIFs. However, due to its complicated asymptotic covariance structure (lacking independent increments) additional techniques for developing 
executable inference procedures are needed. As outlined in the next section, this can either be attacked by computing the corresponding critical values via valid bootstrap procedures
or, alternatively, by approximation techniques for approaching the asymptotic distribution up to a certain degree of accurateness.

%
\section{The Testing Procedures}\label{sec: tests}
%

\subsection{The Test Statistics}\label{subsec: tests}

Let now $I = [t_1,t_2] \subseteq [0,\tau), t_1 < t_2,$ be the interval on which we are interested to compare the CIFs $F_1^{(1)}$ and $F_1^{(2)}$.
There are plenty of possible test statistics for testing the hypotheses~\eqref{eq:hypotheses}
which can be based on $W_{n_1,n_2}$. 
The main idea is to plug the process $W_{n_1,n_2}$ into continuous functionals $\phi: \D[t_1,t_2] \rightarrow [0,\infty)$
so that $\phi(W_{n_1,n_2})$ tends to infinity for $\min(n_1,n_2) \rightarrow \infty$ and $\frac{n_1}{n} \rightarrow p$, 
whenever the alternative hypothesis $H_{\neq}$ is true. On the other hand, $\phi(W_{n_1,n_2})$ should converge to a non-degenerated limit on $H_=$. We here only discuss two possibilities and refer to connected literature on goodness-of-fit testing for further examples. 
As already suggested in Bajorunaite and Klein (2007) one possibility is to consider a weighted version of Kolmogorov-Smirnov-type, i.e.,
\begin{align}\label{eq:KS_Stat}
 T^{KS}=\sup\limits_{u \in [t_1,t_2]} \rho_1(u) | W_{n_1,n_2}(u) |,
\end{align}
where $\rho_1: [t_1,t_2] \rightarrow (0,\infty)$ is some measurable and bounded weight function. 
Another choice is a weighted version of a two-sample Cram\'{e}r-von Mises-type statistic, i.e.,
\begin{align}\label{eq:CvM_Stat}
 T^{CvM}=\int_{t_1}^{t_2} \rho_2 (u) W^2_{n_1,n_2}(u) \d u,
\end{align}
where now $\rho_2: [t_1,t_2] \rightarrow (0,\infty)$ is a measurable and integrable weight function.
The asymptotic distribution of these statistics can immediately be obtained from the weak convergence results for $W_{n_1,n_2}$ stated in Theorem~\ref{thm:sqrt_n_conv_cifs} and applications of the continuous mapping theorem.
\begin{thm}
\label{thm:test statistics}
 Under the conditions and notation of Theorem~\ref{thm:sqrt_n_conv_cifs} the convergences in distribution
 \begin{align}\label{eq: conv TKS}
  T^{KS}  \oDo \sup\limits_{u \in [t_1,t_2]} \rho_1(u) | V(u) |
 \end{align}
 \begin{align}\label{eq: conv TCvM}
   T^{CvM} \oDo \int_{t_1}^{t_2} \rho_2 (u) V^2(u) \d u
 \end{align}
 hold true.
Moreover, if $\rho_2$ is even continuous, the following representation in distribution holds for the limit in \eqref{eq: conv TCvM}
 \begin{align}\label{eq: dist wchi2}
 \int_{t_1}^{t_2} \rho_2 (u) V^2(u) \d u \gDg \sum_{j=1}^\infty \lambda_j Z_j^2,
\end{align}
where $(Z_j)_j$ are i.i.d. standard normal random variables and $(\lambda_j)_j$ are the eigenvalues of the covariance function $\zeta_{\rho_2}(s_1,s_2)=\rho_2^{1/2}(s_1)\zeta_V(s_1,s_2)\rho_2^{1/2}(s_2)$;
see~\eqref{eq:eigen} in the Appendix for details.
 \end{thm}

\begin{rem}
\label{rem:test statistics}
  \text{ } {\\}
  (a) In general, the above test statistics cannot be made asymptotically pivotal by any transformation, 
  so that there is no obvious way 
    to state a valid asymptotical test in the classical sense.\\
  (b) Note that a Pepe (1991) \nocite{Pepe91} type statistic, $\int_{t_1}^{t_2} \rho_2(u) W_{n_1,n_2}(u) \d u$, actually 
  leads to a test for ordered CIFs, i.e., for the null hypothesis
    \begin{align*}
     H_\leq : \{ F_1^{(1)} \leq F_1^{(2)} \text{ on } [t_1,t_2] \} 
      \; \text{ versus } \;
     H_{\gneqq} : \{ F_1^{(1)} \geq F_1^{(2)} \text{ on } [t_1,t_2] \; \& \; F_1^{(1)} \neq F_1^{(2)} \}.
    \end{align*}
    {\color{black} Hence, although frequently used for testing equality of CIFs ($H_=$), it can only detect alternatives of the form $H_{\gneqq}$ and possesses 
    low power for other alternatives. A similar comment applies for Gray's (1988) test\nocite{Gray_1988} which is also quite popular in practice. 
    In particular, Gray's test statistic is based on integrating differences of weighted \emph{increments} of CIFs, and is thus not 
    consistent against all alternatives of the form $H_{\neq}$. Nevertheless we will consider both tests in our simulations; see Section~\ref{sec: sim} below.}\\
    In Bajorunaite and Klein (2007, 2008) \nocite{bajorunaite07} and Dobler and Pauly (2014) \nocite{dobler14} tests of Pepe-type have been utilized for testing 
    $H_\leq$ versus $H_{\gneqq}$ in combination with Lin's (1997) \nocite{lin97} and Efron's (1979) \nocite{efron79} resampling techniques, respectively.\\
  (c) Choices for $\rho_i$: For simplicity, we could take $\rho_i \equiv 1$. In contrast, the weight function
    \begin{align*}
     \rho_2(u) = \frac{1}{\sqrt{(t_2 - u) (u - t_1)}}
    \end{align*}
    corresponds to an Anderson-Darling-type test for CIFs. In this case, however, the representation~\eqref{eq: dist wchi2} no longer holds.\\
    Moreover, it can also be shown that the asymptotic results~\eqref{eq: conv TKS}--\eqref{eq: dist wchi2} hold for data-dependent weight functions $\hat \rho_i$ as long as
    $\hat \rho_i \oPo \rho_i$ uniformly on $[t_1,t_2]$ in probability with $\rho_i: [t_1,t_2] \rightarrow (0,\infty)$ measurable and bounded (for $i=1$) or integrable (for $i=2$) and continuous (for the representation of $T^{CvM}$).\\
\end{rem}

Due to the asymptotic non-pivotality of these test statistics,
critical values of the correspon-ding tests cannot be assessed directly form their asymptotics. 
In the following we therefore introduce different approaches for calculating critical values that lead to adequate test decisions.

\subsection{Wild Bootstrap Tests}\label{subsec:bs}

For the computation of critical values, we start by formulating a wild bootstrap statistic
which has the same asymptotic distribution as $W_{n_1,n_2}$ under $H_=$.
To this end, consider a linear martingale representation of $W_{n_1,n_2}$,
\begin{align*}
 W_{n_1,n_2}(s) & = \sqrt{\frac{n_1 n_2}{n}} \sum_{k=1}^2 (-1)^{k+1} \sum_{i=1}^{n_k}
  \Big\{ \int_0^s \frac{ S_2^{(k)}(u) - F_1^{(k)}(s) }{Y^{(k)}(u)} \d M_{1;i}^{(k)}(u) \\
  & \quad + \int_0^s \frac{ F_1^{(k)}(u) - F_1^{(k)}(s)}{Y^{(k)}(u)} \d M_{2;i}^{(k)}(u) \Big\}
  + o_p(1);
\end{align*}
see also Lin (1997) in the case of solely right-censored data and Beyersmann et al. (2013), Dobler and Pauly (2014) or Dobler et al. (2015) for more general situations.
Now, Lin's resampling technique is based on replacing all unknown CIFs by their Aalen-Johansen estimators
and each $\d M_{j;i}^{(k)}$ with $G^{(k)}_{j;i} \d N_{j;i}^{(k)}$,
where the $G^{(k)}_{j;i}$ are i.i.d. standard normal variates, independent of the data.
This leads to the wild bootstrap statistic
\begin{align*}
 \hat W_{n_1,n_2}(s) & = \sqrt{\frac{n_1 n_2}{n}} \sum_{k=1}^2 (-1)^{k+1} \sum_{i=1}^{n_k}
  \Big\{ \int_0^s \frac{ \hat S_2^{(k)}(u) - \hat F_1^{(k)}(s) }{Y^{(k)}(u)} G^{(k)}_{1;i} \d N_{1;i}^{(k)}(u) \\
  & \quad + \int_0^s \frac{ \hat F_1^{(k)}(u) - \hat F_1^{(k)}(s) }{Y^{(k)}(u)} G^{(k)}_{2;i} \d N_{2;i}^{(k)}(u) \Big\} .
\end{align*}
Beyersmann et al. (2013) generalized this approach by allowing the $G^{(k)}_{j;i}$ to be
i.i.d. zero-mean random variables with variance 1 and finite fourth moment.
This resampling scheme was further extended by Dobler et al. (2015) to even allow for conditionally independent,
data-dependent multipliers $G^{(k)}_{j;i}$.
Beyersmann et al. (2013) proved a conditional limit theorem for a one-sample version of $\hat W_{n_1,n_2}$
 from which we can directly deduce the following result.
\begin{thm}[Beyersmann et al. (2013)]
 Suppose~\eqref{eq:at_risk_conv} holds for both sample groups on the interval $[t_1,t_2]$.
 Conditioned on the data convergence in distribution
 \begin{align*}
  \hat W_{n_1,n_2} \oDo V
 \end{align*}
 holds on the Skorohod space $\D[t_1,t_2]$ in probability under both $H_=$ as well as $H_{\neq}$.
 Here $V$ is a time-continuous, zero-mean Gaussian process with covariance function given by~\eqref{eq:zeta_V}.
\end{thm}
Since $W_{n_1,n_2}$ and its wild bootstrap version $\hat W_{n_1,n_2}$ have the same limit under $H_=$, the construction of asymptotic level $\alpha$ tests is now accomplished 
by also plugging $\hat W_{n_1,n_2}$ into the corresponding continuous functionals $\phi$. 
Consequently, the resulting tests depending on $\phi(W_{n_1,n_2})$ (as test statistics) and $\phi(\hat W_{n_1,n_2})$ (yielding data-dependent critical values) are asymptotic level $\alpha$ tests. 
Furthermore, the tests are consistent, that is, they reject the alternative hypothesis $H_{\neq}$ with probabilities tending to 1 as $n \rightarrow \infty$.
Thus, the following theorem follows immediately from the weak convergence results of the preceding theorems for $W_{n_1,n_2}$ and $\hat W_{n_1,n_2}$ and from applications of the continuous mapping theorem.
\begin{thm} \label{thm:resampling_wild_cif_equality_1}
 Let $G_{j;i}^{(k)}, i=1,\dots,n_k \in \N, j,k=1,2,$ be i.i.d. zero-mean wild bootstrap weights
 with existing fourth moments and variance 1. Then the following tests are asymptotic level $\alpha$ wild bootstrap tests for $H_=$ vs. $H_{\neq}$:
  \begin{align*}
    \varphi^{KS} = \left\{
    \begin{array}{rlcr}
      1 &	& > 	&\\
	& T^{KS} & 	& c^{KS} \\
      0	&	& \leq 	&
    \end{array}
    \right.,\quad
\varphi^{CvM} = \left\{
    \begin{array}{rlcr}
      1 && > & \\
	&	T^{CvM}	& & c^{CvM}  \\
      0	&& \leq &
    \end{array}
    \right.,
  \end{align*}  
  where $c^{KS}( \cdot )$ and $c^{CvM}( \cdot )$ are the data-dependent $(1-\alpha)$-quantiles of the conditional distributions
  of $\sup_{u \in [t_1,t_2]} \rho_1(u) | \hat W_{n_1,n_2}(u) |$ and $\int_{t_1}^{t_2} \rho_2(u) \hat W^2_{n_1,n_2}(u) \d u$, respectively, given the observations.
\end{thm}
\begin{rem}
  \text{ } {\\}
  (a) The above bootstrap version of the Kolmogorov-Smirnov-type test statistic has also been suggested in Bajorunaite and Klein (2007)\nocite{bajorunaite07}.
    The present article additionally provides a theoretical justification of its large-sample properties.\\
  (b) In general, the exchangeably weighted bootstrap discussed in Dobler and Pauly (2014) \nocite{dobler14} 
  is not applicable since the wrong limiting covariance structure of the bootstrapped process leads to an asymptotically incorrect critical value.\\
  (c) A modification of Theorem~\ref{thm:resampling_wild_cif_equality_1} can be utilized 
    for the construction of asymptotically valid confidence bands for $F_1^{(1)} - F_1^{(2)}$;
    see Beyersmann et al. (2013) for further details with regard to the one-sample case.\\
  (d) Also in this case, it can be shown that the results hold for data-dependent weight functions $\hat \rho_i$ as long as
    $\hat \rho_i \oPo \rho_i$ uniformly on $[t_1,t_2]$ in probability with $\rho_i$ as in Theorem~\ref{thm:test statistics}. 
   For example, it would be possible to choose $\hat \rho_2$ as a kernel density estimator for $\rho_2 = (1-p) \alpha_1^{(1)} + p \alpha_1^{(2)}$ 
    if both cause-specific hazard intensities are continuous. Here the kernel function needs to be of bounded variation and the bandwidth $b_n \rightarrow 0$ may fulfill $\sup_{u \in [t_1,t_2]} ( b_n^2 Y^{(n_k)}(u) )^{-1} \oPo 0$, $k=1,2$.
    For more details, see Section~IV.2 in Andersen et al. (1993). Similarly, other goodness-of-fit statistics may be realized.\\
 (e) Note that the case with only one competing risk yields wild bootstrap versions of classical goodness-of-fit tests.
\end{rem}

In practical situations critical values are calculated by Monte-Carlo simulations, repeatedly generating standardized wild bootstrap weights; see e.g., Lin (1997) or Beyersmann et al. (2013) for additional details.

%
\subsection{Approximation Procedures}\label{subsec:approx_proc}
%

In case of the Cram\'{e}r-von Mises statistic with continuous $\rho_2$,
another way to approximate the unknown asymptotic $(1-\alpha)$-quantile 
of Theorem~\ref{thm:test statistics} (under the null hypothesis of equal CIFs for the first risk) may be based on a Box or Pearson approximation; see Box (1954) and Pearson (1959) as well 
as Rauf Ahmad et al. (2008) \nocite{Ahmad_2008} or Pauly et al. (2013) \nocite{PEB13} for applications of these approaches for inference of high-dimensional data.

The main idea is to approximate the distribution of 
\begin{align}
  Q=\sum_{j=1}^\infty \lambda_j Z_j^2, 
\end{align}
the limit distribution of $T^{CvM}$, by adequately transformed $\chi^2$-distributions. 
In case of the {\it Box approximation} this is done by equating the first two moments of $Q$ with those of a scaled $g \chi_f^2$-distribution. 
Recall that the expected value and variance of $g \chi_f^2$ are given by $E[g \chi_f^2] = gf$ {and} $Var(g \chi_f^2) = 2 g^2 f,$ respectively. 
Thus, $f,g$ need to solve the following equations for matching the first two asymptotic moments of the test statistic $T^{CvM}$:
  \begin{align}\label{eq: eqBox1}
   gf & = \E [ Q ] = \int \rho_2(u) \zeta_V(u,u) \d u =\mu\\
   \label{eq: eqBox2}\quad \text{and} \quad 2 g^2 f & = Var ( Q ) 
   = 2 \int \int \rho_2(u) \zeta_V^2(u,s) \rho_2(s) \d u \d s =\sigma^2
  \end{align}
where the integrals run over the interval $[t_1,t_2]$.  The justification for exchanging the order of integration is given in the Appendix; see the proof of Theorem~\ref{thm:test statistics}. 
This leads to the choices
\begin{align*}
 f = \frac{2 \mu^2}{\sigma^2} \quad \text{and} \quad g = \frac{\sigma^2}{2 \mu}
\end{align*}
which fulfill $\E[g \chi_f^2] = \E[Q]$ and $Var(g \chi_f^2) = Var(Q).$

Since $f$ and $g$ are in general unknown, adequate consistent estimators are needed. This is achieved via plugging in the canonical Welch-type covariance estimator 
\begin{align}\label{eq:zeta_hat}
 \hat \zeta_{n_1,n_2} = \frac{n_2}{n} \hat \zeta_{n_1}^{(1)} + \frac{n_1}{n} \hat \zeta_{n_2}^{(2)} 
\end{align}
with
  \begin{align}
\nonumber   & \hat \zeta^{(k)}_{n_k}(s_1,s_2) = n_k \int_0^{s_1 \wedge s_2} \frac{ \{ \hat S_2^{(k)}(u) - \hat F_1^{(k)}(s_1) \} 
   \{ \hat S_2^{(k)}(u) - \hat F_1^{(k)}(s_2) \} } {(Y^{(k)})^2(u)} \d N_{1}^{(k)}(u) \\
\label{eq: zeta_hat_k}   & \quad  + n_k \int_0^{s_1 \wedge s_2} \frac{ \{ \hat F_1^{(k)}(u) - \hat F_1^{(k)}(s_1) \} 
   \{ \hat F_1^{(k)}(u) - \hat F_1^{(k)}(s_2) \} } {(Y^{(k)})^2(u)} \d N_{2}^{(k)}(u).
  \end{align}
In the Appendix it is shown that $\hat \zeta_{n_1,n_2}$ is uniformly consistent on the rectangle $[t_1,t_2]^2$.
The resulting Box-type approximation is summarized as a theorem.

\begin{thm}[A Box-type approximation]
\label{thm:box_approx}
 Let $\rho_2: [t_1,t_2] \rightarrow (0,\infty)$ be a continuous weight function. 
Then 
  \begin{align*}
   \hat f := \frac{2 \hat \mu_{n_1,n_2}^2}{\hat \sigma_{n_1,n_2}^2} \quad \text{and} \quad 
  \hat g := \frac{\hat \sigma_{n_1,n_2}^2}{2 \hat \mu_{n_1,n_2}}
  \end{align*}
are consistent estimators for $f,g > 0$ such that $\E[g \chi_f^2] = \E[Q]$ and $Var(g \chi_f^2) = Var(Q)$. Here 
\begin{align*}
    & \hat \mu_{n_1,n_2} := \int_{t_1}^{t_2} \rho_2(s) \hat \zeta_{n_1,n_2}(s,s) \d s \text{ and }
    \hat \sigma_{n_1,n_2}^2 := 2 \int_{[t_1,t_2]^2} \rho_2(s_1) \hat \zeta_{n_1,n_2}^2(s_1,s_2) \rho_2(s_2) \d\leb^2(s_1,s_2).
\end{align*}
are consistent estimators for the asymptotic mean and variance of $T^{CvM}$, respectively.
\end{thm}
Following Box (1954), we can deduce an approximative test for $H_=$ vs. $H_{\neq}$ by 
  \begin{align*}
    \varphi^{B} = \left\{
    \begin{array}{rlcr}
      1 && > &  \\
	&	T^{CvM}	& & c^{B} \\
      0	&& \leq &
    \end{array}
    \right.
  \end{align*}
  where $c^{B}( \cdot )$ is the $(1-\alpha)$-quantile of $\hat g \chi_{\hat f}^2$.

For an extension of this approach one might think about matching even more moments; see e.g., Pauly et al. (2013) for an application and additional motivation. As in that paper we now consider 
a studentized version of the test statistic given by
$$
  T^{CvM}_{\textnormal{stud}} = \frac{T^{CvM}-\hat \mu_{n_1,n_2}}{\hat \sigma_{n_1,n_2}}
$$
with $\hat \mu_{n_1,n_2}$ and $\hat \sigma^2_{n_1,n_2}$ as in Theorem~\ref{thm:box_approx}. Its asymptotic distribution is given by the law of
\begin{align*}
      Q_{\textnormal{stud}} := \frac{Q - \mu}{\sigma} := \frac{Q - \E[Q]}{Var(Q)^{1/2}}
\end{align*}
with $\mu = \sum_{j=1}^\infty \lambda_j$ and $\sigma^2 = 2 \sum_{j=1}^\infty \lambda_j^2$. 
This follows from Theorem~\ref{thm:test statistics} and the consistency of 
$\hat \mu_{n_1,n_2}$ and $\hat \sigma^2_{n_1,n_2}$ for $\mu$ and $\sigma^2$ {\color{black}are} shown in the proof of Theorem~\ref{thm:box_approx}. 
Now, the {\it Pearson approximation} 
of the distribution of $Q_{\textnormal{stud}}$ is the law of the random variable 
    \begin{align*}
      \chi_{\kappa,\textnormal{stud}}^2 := \frac{\chi_\kappa^2 - \E[\chi_\kappa^2]}{Var(\chi_\kappa^2)^{1/2}} 
	= \frac{\chi_\kappa^2 - \kappa}{\sqrt{2\kappa}}.
    \end{align*}
Here the parameter $\kappa$ is chosen in such a way that mean, variance and skewness of $\chi_{\kappa,{\textnormal{stud}}}^2$ and $Q_{\textnormal{stud}}$ coincide. 
As shown in the proof of Theorem~\ref{thm:pearson_approx}, this leads to the choice
\begin{align*}
    \kappa =  \frac{\left(\sum_{j=1}^\infty \lambda_j^2\right)^{3}}{\left(\sum_{j=1}^\infty \lambda_j^3 \right)^{2}}.
\end{align*}
Since the parameter $\kappa > 0$ is unknown, it needs to be estimated.
The resulting Pearson approximation is summarized below.
\begin{thm}[A Pearson-type approximation]
\label{thm:pearson_approx}
 Let $\rho_2: [t_1,t_2] \rightarrow (0,\infty)$ be a continuous weight function.
 Then the estimator
    \begin{align*}
      \hat \kappa := \frac{\hat \sigma_{n_1,n_2}^6}{8 \hat \gamma^2_{n_1,n_2} }
    \end{align*}
is consistent for the true parameter $\kappa$ that leads to the desired equalities of mean, variance and skewness of $Q_{\textnormal{stud}}$ and $\chi_{\kappa,\textnormal{stud}}^2$. Here 
    \begin{align*}
      \hat \gamma_{n_1,n_2} := \int_{[t_1,t_2]^3} 
	 \rho_2(s_1) \hat \zeta_{n_1,n_2}(s_1,s_2)  \rho_2(s_2) \hat \zeta_{n_1,n_2}(s_2,s_3)
	 \rho_2(s_3) \hat \zeta_{n_1,n_2}(s_3,s_1) \d\leb^3(s_1,s_2,s_3)
    \end{align*}
is a consistent estimator for $\sum_{j=1}^\infty \lambda_j^3$.
\end{thm}
Following Pearson (1959), an approximative test for $H_=$ vs. $H_{\neq}$ is given by
  \begin{align*}
    \varphi^{P} = \left\{
    \begin{array}{rlcr}
      1 && > & \\
	&	T_{\textnormal{stud}}^{CvM}	& & c^{P}  \\
      0	&& \leq &
    \end{array}
    \right.
  \end{align*}
  where $c^{P}( \cdot )$ is the $(1-\alpha)$-quantile of $\chi_{\hat{\kappa},\textnormal{stud}}^2.$

Since the Pearson-type approximation additionally matches the skewness in the limit, it is expected 
to be the superior to the Box-type approximation. 
However, the additional parameter estimator $\hat \gamma_{n_1,n_2}$ 
may also lead to a greater inaccuracy. 
In order to check the tests' actual performances, 
we investigate both approximation procedures and the wild bootstrap tests in the next section.

%
\section{Simulations}\label{sec: sim}
%

The previous section coped with three kinds of statistical tests for the hypotheses $H_=$ vs. $H_{\neq}$:
\begin{enumerate}
  \item Asymptotically (as $n \rightarrow \infty$) consistent tests using wild bootstrap techniques.
  \item Approximative tests mimicking the asymptotic distribution 
  of the Cram\'er-von Mises test statistic while estimating the relevant parameters.
  \item A possibly asymptotically inconsistent Pepe-type test using wild bootstrap techniques.
\end{enumerate}
All methods intend to give good small sample results with regard to level $\alpha$ control,
while the wild bootstrap tests shall clearly outperform the approximative tests for sample sizes going to infinity.
This is due to the approximative nature of those tests; their critical values will not be exact in the limit.
On the other hand, a good approximation might yield critical values 
close to the actual quantiles of the test statistic --
if the involved point estimators are reliable.
In this case it is conceivable that the approximative tests may outperform the wild bootstrap tests.
Keeping the type-I error rate in mind, 
we are further interested in the small sample power of the above tests.
{\color{black}To have another asymptotic reference test, we followed the suggestion of a referee to also include Gray's (1988) test.}

To investigate the actual small sample behaviour of all considered tests,
we consider the following {\color{black}two} set-ups: 
Each simulation was carried out utilizing 
the R-computing environment, version 2.15.0 (R Development Core Team, 2010)
with $N_{sim} = 1000$ simulation runs.
Additionally, all resampling tests were established with $B=999$ bootstrap runs in each of the $N_{sim}$ steps.
For all of the following set-ups the nominal size is $\alpha = 5\%$.
\begin{enumerate}
 \item Model 1 of Bajorunaite and Klein (2007) is given by the CIFs for both risks as
  \begin{align*}
   F_1^{(k)}(t) = \frac{p (1 - e^{-t} ) e^{\beta Z}}{1 - p + p ( 1 - e^{-t} ) e^{\beta Z} }
   \quad \text{and} \quad F_2^{(k)}(t) = \frac{(1-p) (1 - e^{-t} )}{1 - p + p e^{\beta Z} }
  \end{align*}
  where $Z = 1$ for sample $k=2$ and $Z=0$ for sample $k=1$.
  In the case $\beta = 0$ the authors pointed out that 
  all CIFs and cumulative hazard functions are equal among both groups
  and thus the null hypothesis is implied.
  The alternative is for example true in our simulations for $\beta = 0.75$.
  The parameter $p \in (0,1)$ specifies the proportion of type 1 events and,
  following Bajorunaite and Klein again, have chosen to be $p = 0.25, 0.5, 0.75$ for the null hypothesis
  and $p = 0.18, 0.41, 0.68$ for the alternative.
  The data is independently right-censored by several different uniform $U(a,b)$-censoring distributions with $a < b$ chosen in a way
  to have $0\%, 25\%$ or $50\%$ censoring -- in the case of untruncated data; 
  see Table 1 of Bajorunaite and Klein (2007) for details.
  We, however, also let $75\%$ of all individuals of each group be independently left-truncated 
  by a gamma-truncation distribution with scale parameter 1.5 and shape parameter 0.75;
  see Beyersmann et al. (2013) for a similar set-up.
  The simulated sample sizes are chosen as $(n_1,n_2) = (20,20), (50,50), (50,100), (100,50),(100,100), (200,200)$
  and the event-time interval of interest is $[t_1,t_2] = [0,3]$.
  {\color{black}Gray's (1988) test has not been included in this part of the simulation study
  since its theory has not been developed for left-truncated data.}
 \item A set-up with crossing CIFs for the first risk is given by Model 2 of Bajorunaite and Klein (2007)
  with $\beta = \log 3, p^{(1)} = 0.42, p^{(2)} = 0.58$ and both competing risks and sample groups interchanged, that is,
  \begin{align*}
   F_1^{(k)}(t) = (1 - p^{(k)}) (1 - e^{-t})^{e^{\beta Z}} \quad \text{and} \quad
   F_2^{(k)}(t) = p^{(k)} (1 - e^{-t}).
  \end{align*}
  For these simulations we have complete observations and the same sample group sizes as in our first simulation set-up.
  As interval of interest we chose $[t_1,t_2]$ approximately equal to $[0,3.8]$ 
  so that the integral of $F^{(1)}(t) - F^{(2)}(t)$ over this area is close to zero.
  In this case we expect that the statistic of the Pepe-type test tends to have very small values
  resulting in a poor power in this set-up.
\end{enumerate}
{\color{black} Additionally, a third simulation design, adopted from Dobler and Pauly (2014), 
is conducted in the supplementary material; see Appendix~\ref{sec: appA}.}
Remember that $75\%$ of all individuals are left-truncated in the first set-up.
This might lead to extremely small sample groups 
so that most of the tests do not keep the nominal size $\alpha = 5\%$ by far, especially if censoring is present, too.
In each simulation run with no event of the first risk we did not reject the null hypothesis.
See Table~\ref{table:niveau_BK_1} for the results which are commented subsequently.
In these simulations for the null hypothesis in the first set-up we observe
that the Pepe-type test $\varphi^{Pepe}$ achieved the type-I error probabilities closest to the nominal size $\alpha = 5\%$, 
especially for stronger censoring scenarios.
Apart from that, the approximative tests $\varphi^{P}$ and $\varphi^{B}$ tend to be closer to the nominal level 
compared to the related Cram\'er-von Mises test $\varphi^{CvM}$.
Overall, however, the differences between those three tests do not really stand out.
The Kolmogorov-Smirnov test $\varphi^{KS}$ shows the largest deviation from the nominal level in most cases,
whereas this test outperforms the tests based on the Cram\'er-von Mises statistic in some situations.

Keeping these results in mind we now compare the achieved powers in the first set-up; see Table~\ref{table:power_BK_1}.
In this case $\varphi^{Pepe}$ utterly disappoints by having even lower rejection rates than under the null.
We observe the same behaviour for the remaining tests, but these situations seem to be exceptional.
All in all, these four tests have satisfactory power (having the small samples in mind).
In most scenarios $\varphi^{KS}$ achieved the greatest power by far 
which, however, may be explained by the quite liberal behaviour under the null.
Again, $\varphi^{P}, \varphi^{B}$ and $\varphi^{CvM}$ differ not that much,
where the latter shows a slight tendency for a greater power compared to the approximative tests. 
{\color{black} It should be pointed out that the increase of the power with larger censoring rate in Table~\ref{table:power_BK_1} is due to 
the larger (true) nominal level in this situation; see Table~\ref{table:niveau_BK_1}.

Recall that Gray's (1988) test was not applied for the first setting since it is not designed for truncated data. 
This is different in the second scenario of } 
crossing CIFs and complete observations. Here the first four tests {\color{black}possess} a similar power behaviour; see Table~\ref{table:power_BK_2}:
$\varphi^{KS}$ has the greatest and $\varphi^{CvM}$ has the second greatest power, 
shortly followed by $\varphi^{B}$ and then $\varphi^{P}$.
As expected, the Pepe-type test does not possess the power to detect this alternative 
since the integral in the statistic cancels out the differences between the Aalen-Johansen estimators. In particular, its 
power is close to the nominal level. 
{\color{black}A similar lack of power is observed for Gray's test $\varphi^{Gray}$ for smaller samples. 
In contrast to the test of Pepe-type, it at least increases for larger sample sizes up to $n_1 = n_2 = 200$. 
However, the difference in terms of power in comparison to the first four tests is unacceptable.}


\begin{sidewaystable}
\begin{center}
\setlength{\tabcolsep}{5pt}
\begin{tabular}[h]{c|c|c|c|c|c|c|c|c|c|c|c|c|c|c|c|c}
\hline
 \multicolumn{2}{r|}{Censoring} & \multicolumn{5}{|c|}{0 \%} & \multicolumn{5}{|c|}{25 \%} & \multicolumn{5}{|c}{50 \%}  \\
 $(n_1,n_2)$ & $p$  & $\varphi^{KS}$ & $\varphi^{CvM}$ & $\varphi^{P}$ & $\varphi^{B}$ & $\varphi^{Pepe}$ & $\varphi^{KS}$ & $\varphi^{CvM}$ & $\varphi^{P}$ & $\varphi^{B}$ & $\varphi^{Pepe}$ & $\varphi^{KS}$ & $\varphi^{CvM}$ & $\varphi^{P}$ & $\varphi^{B}$ & $\varphi^{Pepe}$ \\\hline
	  & 0.25 & .076 & .066 & \bf .065 & \bf .065 & .074 & .105 & .096 & .095 & .095 & \bf .066 & .183 & .159 & .159 & .158 & \bf .116 \\
$(20,20)$ & 0.50 & .083 & .084 & .085 & .085 & \bf .062 & .089 & .099 & .093 & .093 & \bf .069 & .181 & .157 & .155 & .154 & \bf .094 \\
	  & 0.75 & \bf .091 & .101 & .100 & .099 & \bf .091 & .097 & .112 & .111 & .112 & \bf .080 & .213 & .246 & .242 & .242 & \bf .133 \\\hline
	  & 0.25 & \bf .056 & \bf .044 & .043 & .043 & .040 & .084 & .060 & .063 & .063 & \bf .052 & .195 & .139 & .139 & .139 & \bf .094 \\
$(50,50)$ & 0.50 & .055 & .054 & .052 & \bf .051 & .053 & .086 & .077 & .077 & .077 & \bf .072 & .180 & .137 & .133 & .133 & \bf .084 \\
	  & 0.75 & .096 & .096 & .093 & .093 & \bf .078 & .090 & .132 & .126 & .126 & \bf .083 & .208 & .207 & .202 & .202 & \bf .124 \\\hline
	  & 0.25 & .079 & .070 & .070 & .070 & \bf .035 & .106 & .070 & .070 & .070 & \bf .039 & .203 & .156 & .158 & .158 & \bf .078 \\
$(50,100)$ & 0.50 & .056 & .048 & \bf .049 & \bf .049 & .038 & .101 & .086 & .081 & .081 & \bf .056 & .216 & .166 & .166 & .166 & \bf .095 \\
	  & 0.75 & \bf .073 & .089 & .087 & .086 & .074 & \bf .080 & .130 & .127 & .127 & .110 & .217 & .212 & .209 & .209 & \bf .138 \\\hline
	  & 0.25 & .064 & \bf .050 & .051 & .051 & .075 & .104 & \bf .069 & .070 & .070 & .086 & .214 & .154 & .150 & .150 & \bf .131 \\
$(100,50)$ & 0.50 & .056 & .056 & .055 & \bf .054 & .057 & .087 & .061 & \bf .060 & \bf .060 & .077 & .212 & .150 & .150 & .150 & \bf .118 \\
	  & 0.75 & .084 & .088 & .089 & .089 & \bf .064 & .089 & .129 & .122 & .122 & \bf .084 & .218 & .222 & .221 & .221 & \bf .147 \\\hline
	  & 0.25 & .059 & \bf .049 & .045 & .045 & .041 & .096 & .059 & \bf .053 & \bf .053 & .045 & .183 & .117 & .116 & .116 & \bf .070 \\
$(100,100)$ & 0.50 & \bf .048 & .056 & .055 & .055 & .047 & .075 & .075 & .071 & .071 & \bf .068 & .202 & .168 & .157 & .155 & \bf .105 \\
	  & 0.75 & .073 & .080 & .081 & .081 & \bf .064 & \bf .093 & .128 & .129 & .129 & .099 & .215 & .222 & .224 & .224 & \bf .155 \\\hline
	  & 0.25 & .061 & .054 & .053 & \bf .052 & .046 & .106 & .064 & .058 & .058 & \bf .052 & .175 & .104 & .102 & .102 & \bf .075 \\
$(200,200)$ & 0.50 & .043 & .047 & \bf .048 & \bf .048 & .053 & .090 & .077 & \bf .070 & \bf .070 & \bf .070 & .168 & .115 & .111 & .111 & \bf .079 \\
	  & 0.75 & .063 & .069 & .068 & .068 & \bf .054 & .099 & .122 & .120 & .120 & \bf .089 & .210 & .226 & .223 & .223 & \bf .137 \\
\end{tabular}
\caption{(Model 1 of Bajorunaite and Klein (2007))
  Simulated sizes of the resampling tests $\varphi^{KS},\varphi^{CvM},\varphi^{Pepe}$
  and the approximative tests $\varphi^P, \varphi^B$ for nominal size $\alpha=5\%$ 
  under different sample sizes and right-censoring distributions and a fixed left-truncation distribution 
  under $H_{=}$.}
\label{table:niveau_BK_1}
\end{center}
\end{sidewaystable}


\begin{sidewaystable}
\begin{center}
\setlength{\tabcolsep}{5pt}
\begin{tabular}[h]{c|c|c|c|c|c|c|c|c|c|c|c|c|c|c|c|c}
\hline
 \multicolumn{2}{r|}{Censoring} & \multicolumn{5}{|c|}{0 \%} & \multicolumn{5}{|c|}{25 \%} & \multicolumn{5}{|c}{50 \%}  \\
 $(n_1,n_2)$ & $p$  & $\varphi^{KS}$ & $\varphi^{CvM}$ & $\varphi^{P}$ & $\varphi^{B}$ & $\varphi^{Pepe}$ & $\varphi^{KS}$ & $\varphi^{CvM}$ & $\varphi^{P}$ & $\varphi^{B}$ & $\varphi^{Pepe}$ & $\varphi^{KS}$ & $\varphi^{CvM}$ & $\varphi^{P}$ & $\varphi^{B}$ & $\varphi^{Pepe}$ \\\hline
	  & 0.18 & .089 & .076 & .079 & .079 & .044 & .078 & .075 & .074 & .074 & .038 & .178 & .179 & .175 & .175 & .078 \\
$(20,20)$ & 0.41 & .090 & .078 & .081 & .081 & .047 & .112 & .108 & .111 & .111 & .042 & .202 & .186 & .184 & .184 & .088 \\
	  & 0.68 & .139 & .105 & .104 & .104 & .076 & .145 & .143 & .144 & .144 & .089 & .234 & .241 & .241 & .242 & .123 \\\hline
	  & 0.18 & .098 & .092 & .089 & .087 & .017 & .153 & .122 & .117 & .117 & .024 & .234 & .171 & .169 & .168 & .041 \\
$(50,50)$ & 0.41 & .095 & .087 & .084 & .084 & .037 & .113 & .085 & .085 & .085 & .037 & .224 & .187 & .183 & .183 & .065 \\
	  & 0.68 & .137 & .081 & .078 & .078 & .064 & .155 & .134 & .135 & .134 & .072 & .252 & .208 & .200 & .200 & .104 \\\hline
	  & 0.18 & .141 & .121 & .122 & .122 & .004 & .202 & .159 & .158 & .158 & .010 & .330 & .248 & .243 & .243 & .028 \\
$(50,100)$ & 0.41 & .106 & .072 & .075 & .075 & .018 & .129 & .086 & .083 & .083 & .029 & .247 & .181 & .179 & .178 & .060 \\
	  & 0.68 & .161 & .075 & .075 & .075 & .052 & .175 & .126 & .125 & .124 & .058 & .302 & .224 & .218 & .218 & .104 \\\hline
	  & 0.18 & .084 & .084 & .080 & .080 & .025 & .114 & .081 & .079 & .079 & .024 & .217 & .159 & .158 & .158 & .051 \\
$(100,50)$ & 0.41 & .080 & .081 & .082 & .081 & .037 & .098 & .083 & .083 & .083 & .072 & .201 & .143 & .137 & .137 & .072 \\
	  & 0.68 & .142 & .098 & .100 & .100 & .058 & .160 & .112 & .113 & .113 & .064 & .290 & .217 & .211 & .211 & .104 \\\hline
	  & 0.18 & .139 & .122 & .123 & .123 & .005 & .169 & .126 & .123 & .123 & .006 & .330 & .223 & .226 & .226 & .025 \\
$(100,100)$ & 0.41 & .103 & .074 & .073 & .073 & .027 & .136 & .092 & .090 & .090 & .036 & .244 & .171 & .171 & .170 & .055 \\
	  & 0.68 & .214 & .100 & .097 & .097 & .062 & .213 & .120 & .121 & .121 & .060 & .335 & .219 & .216 & .215 & .113 \\\hline
	  & 0.18 & .187 & .172 & .173 & .173 & .004 & .279 & .226 & .228 & .228 & .007 & .385 & .278 & .272 & .272 & .011 \\
$(200,200)$ & 0.41 & .127 & .077 & .077 & .077 & .012 & .157 & .094 & .099 & .099 & .014 & .307 & .207 & .205 & .205 & .042 \\
	  & 0.68 & .332 & .112 & .105 & .105 & .045 & .362 & .168 & .164 & .164 & .050 & .491 & .281 & .273 & .272 & .104 \\
\end{tabular}
\caption{(Model 1 of Bajorunaite and Klein (2007)) 
  Simulated power of the resampling tests $\varphi^{KS},\varphi^{CvM},\varphi^{Pepe}$
  and the approximative tests $\varphi^P, \varphi^B$ for nominal size $\alpha=5\%$ 
  under different sample sizes and right-censoring distributions and a fixed left-truncation distribution 
  under $H_{\neq}$.}
\label{table:power_BK_1}
\end{center}
\end{sidewaystable}

%
%
%
%
%


\begin{table}
\begin{center}
\setlength{\tabcolsep}{5pt}
\begin{tabular}[h]{c|c|c|c|c|c|c}
\hline 
$(n_1,n_2)$	& $\varphi^{KS}$ & $\varphi^{CvM}$ & $\varphi^{P}$ & $\varphi^{B}$ & $\varphi^{Pepe}$ & $\varphi^{Gray}$ \\\hline
$(20,20)$ 	& .241 & .139 & .134 & .134 & .085 & .059 \\
$(50,50)$ 	& .366 & .175 & .171 & .172 & .065 & .094 \\
$(50,100)$ 	& .479 & .259 & .264 & .266 & .056 & .119 \\
$(100,50)$ 	& .535 & .226 & .220 & .220 & .067 & .078 \\
$(100,100)$ 	& .718 & .455 & .436 & .439 & .054 & .166 \\
$(200,200)$ 	& .976 & .913 & .915 & .916 & .050 & .328 \\
\end{tabular}
\caption{(Model 2 of Bajorunaite and Klein (2007) with crossing CIFs)
  Simulated power of the resampling tests $\varphi^{KS},\varphi^{CvM},\varphi^{Pepe}$,
  {\color{black}the approximative tests $\varphi^P, \varphi^B$
  and the asymptotic test $\varphi^{Gray}$}
  for nominal size $\alpha=5\%$ 
  under different sample sizes and complete observations.}
\label{table:power_BK_2}
\end{center}
\end{table}

{\color{black}
The poor power of $\varphi^{Pepe}$ in the first scenario as well as its great power in the third scenario (see the supplement) is easily explained:}
As already mentioned in Remark~\ref{rem:test statistics}, $\varphi^{Pepe}$ is actually a one-sided test for ordered CIFs, 
that is, for the null hypothesis $H_{\leq} : \{ F^{(1)} \leq F^{(2)} \text{ on } [t_1,t_2] \}$
versus $H_{\gneqq} : \{ F^{(1)} \geq F^{(2)} \text{ on } [t_1,t_2] \; \& \; \{ F^{(1)} \neq F^{(2)} \}$;
see also Dobler and Pauly (2014).
This detail is the reason why $\varphi^{Pepe}$ had such a low power in Table~\ref{table:power_BK_1}
and such a great power in Table~\ref{table:power_DP} {\color{black}in Appendix~\ref{sec: appA}}:
The power simulations for the first set-up considered CIFs contained in $H_{\leq}$
whereas those of the third set-up are covered by $H_{\gneqq}$.
A solution for avoiding such problems is given by utilizing a two-sided version of the Pepe-type test,
for example, by taking the absolute value of the statistic.
This would result in a gain of power in Table~\ref{table:power_BK_1}
but presumably in a loss of power in Table~\ref{table:power_DP} {\color{black}in Appendix~\ref{sec: appA}}.
Furthermore, this {\color{black}modified} test would still be unable to detect crossing CIFs as those in the second set-up.
To solve this, one may consider the absolute value in the integrand of Pepe's statistic,
resulting in a statistic very similar to the Cram\'er-von Mises statistic.

All in all, we advise not to choose $\varphi^{Gray}$ {\color{black}nor} $\varphi^{Pepe}$ nor a modifiaction 
as described above for testing $H_=$ against $H_{\neq}$ due to the disability to detect certain types of alternatives as, e.g., crossing CIFs. 
We furthermore also suggest not to use $\varphi^{KS}$ 
since this test indicated a too liberal behaviour,
especially in cases with moderate to strong censoring; see Table~\ref{table:niveau_BK_1}.
The remaining simulation results showed a slight but numerous superiority of the approximative tests over $\varphi^{CvM}$
when it comes to maintain a prescribed level.
This is indicated in the first scenario as well as in the additional simulation study conducted in the supplementary material, that is, 
in an uncensored and moderately censored scenario, as well as in a scenario with heavy censoring and strong truncation, even for small samples.
On the other hand, $\varphi^{CvM}$ has a slightly greater power than the approximative tests
in most cases (but not in all). 
Due to the focus on maintaining a nominal level $\alpha$,
we therefore advise the utilization of $\varphi^{P}$ or $\varphi^{B}$ for testing $H_=$ --
at least for sample sizes up to $n_1 = n_2 = 200$.
Due to the asymptotically consistency of $\varphi^{CvM}$,
this test will eventually have better type-I error probabilities for very large samples. 
However, it is computationally much more expensive than $\varphi^{P}$ and $\varphi^{B}$ due to the involved 999 Bootstrap Monte Carlo steps.
Both approximative tests $\varphi^{P}$ and $\varphi^{B}$ have an almost equally good performance
so that we cannot detect a clear preference for one over the other. 

Finally, it should be mentioned that, for ordered CIF alternatives, directional tests as $\varphi^{Pepe}$ should possess a larger power since they are constructed 
to detect this smaller class of alternatives. This is, e.g., evident in the simulation design given in the supplement, 
where the test of Pepe-type outperforms the others.

%
\section{{\color{black}Example: Application to BSI data}}\label{sec: data_ex}
%

{\color{black}

The presented testing procedures were all applied to a data-set from the prospective multi-centre cohort study ONKO-KISS about bloodstream infection (BSI)
during neutropenia. 
In the original data-set provided in Dettenkofer et al. (2005)~\nocite{dettenkofer05} a total
number of 1,899 patients were included in the study, 
each having undergone a peripheral blood stem-cell transplantation in 18 different hospitals in Austria, Germany and Switzerland. 
The competing risks in this study are BSI during neutropenia and end of neutropenia or death, both without a preceding BSI.
By combining the latter two events into the single endpoint ``no BSI'',
the theory of the present article becomes available.
The data has, e.g., been analyzed by Beyersmann et al. (2007) or Meyer et al. (2007) in the competing risks context, 
where also medical circumstances and a description of the study are given in detail. 
Here we supplement the analysis by studying the following four questions: Are there differences with respect to the CIFs for BSI between
\begin{itemize}
 \item[(i)] allogeneic and autologous transplants,
 \item[(ii)] female and male groups,
 \item[(iii)] allogeneic and autologous transplants among all \emph{female} patients,
 \item[(iv)] allogeneic and autologous transplants among all \emph{male} patients?
\end{itemize}
Due to free availability we only consider a random subsample of 1,000 patients which is, e.g., 
provided in the \verb=R=-package \verb=compeir= via the data-set \emph{okiss}. 
Of this subsample 564 had undergone allogeneic transplantation and 436 had autologous transplants. 
Moreover, 381 patients were of female and 619 of male gender.
Table~\ref{table:proportions} gives the sample sizes of all transplant type- and gender-specific subgroups.
\begin{table}[H]
\begin{center}
  \begin{tabular}{c||ccc|ccc||c} 
    Subgroup & \multicolumn{3}{|c|}{female} & \multicolumn{3}{|c||}{male} &  \\ \hline
    with event & BSI & no BSI & censored & BSI & no BSI & censored & total \\ \hline \hline
    allogeneic 	& 50 (21.8) & 177  (77.3) & 2 (0.9) & 70 (20.9) & 259 (77.3) & 6 (1.8) & 564 \\
    total	& \multicolumn{3}{|c|}{229} & \multicolumn{3}{|c||}{335} & \\ \hline
    autologous 	& 20 (13.2) & 130 (85.5) & 2 (1.3) & 63 (22.2) & 218 (76.8) & 3 (1.1) & 436 \\
    total	& \multicolumn{3}{|c|}{152} & \multicolumn{3}{|c||}{284} & \\ \hline
    total 	& \multicolumn{3}{|c|}{381} & \multicolumn{3}{|c||}{619} & 1,000 \\
  \end{tabular}
  \caption{{\color{black}Sample sizes of subgroups in the \emph{okiss} data-set (and their rounded proportions in each category in per cent)}}
\label{table:proportions}
\end{center}
\end{table}
As time interval of most interest we have chosen the first five weeks after transplantation, i.e., $[t_1,t_2] = [0, 35]$ in days,
since already 98.1 per cent of all event and censoring times are contained within this period.
The available event times are only right-censored (to a minimal degree of 1.3 per cent) and \emph{not} left-truncated.
Thus, we are able to include all of the six tests discussed earlier in the data analysis. 
In order to meet the assumptions of the theory developed earlier in this article
occurring ties in the event times have been broken by adding normally distributed values with a very small standard deviation.

The estimated CIFs for the analyses of subgroups (i) and (ii) are illustrated in Figure~\ref{fig:data_ex1}(a)-(b)
and the $p$ values for all tests comparing those subgroups are given in the table in Figure~\ref{fig:data_ex1}(c).
Due to crossing Aalen-Johansen estimators for the transplant type-specific CIFs it is not surprising that $\varphi^{Pepe}$ and $\varphi^{Gray}$ 
do not detect any difference at nominal level of $\alpha = 5\%$. However, the other tests do also not yield significant results of which the Kolmogorov-Smirnov test provides the smallest $p$ value of $.136$. 
When testing the CIFs for BSI gender-specifically, the Pepe-type test yields the smallest $p$ value of $.087$
which is again expected in a situation where the Aalen-Johansen estimates indicate an alternative of ordered CIFs.
The $p$ values of the remaining tests are more or less of an equal size. In any case the differences of the CIFs in (i) and (ii) seem to be too small to be detected as significantly distinct.  

The situation is much different after first having divided the data-set according to gender 
and then comparing the CIFs for BSI of both transplant types for each of the genders female (iii) and male (iv).
Figure~\ref{fig:data_ex2}(a) indicates a strong difference of these CIFs (and also an ordering of those) among the group of women (iii).
Thus, all tests yield borderline $p$ values, where the Pepe-type test is the only one below the nominal level of $5\%$. 
The group of male patients (iv) does not show such a huge difference between both Aalen-Johansen estimates but also an alternative of ordered CIFs is indicated -- here in the opposite direction.
This explains the large $p$ values of $\varphi^{Pepe}, \varphi^{Gray}$. In comparison, the tests $\varphi^{CvM}, \varphi^{B}, \varphi^{P}$ based on a Cram\'er-von Mises statistic lead to smaller 
(but also not significant) $p$ values around $.2$. Finally, the Kolmogorov-Smirnov test with a $p$ value of $.019$ is the only one with a clearly significant decision for the alternative of unequal CIFs.
This is in line with our findings from Section~\ref{sec: sim}, where the wild bootstrap Kolmogorov-Smirnov test had the largest power of all tests. Here its liberality may not be such a big issue due 
to the large sample size and almost negligible censoring rates.

Note, that a thorough study of the above examined CIFs with thus many comparisons would of course require a multiple testing adjustment.
Our aim, however, was to illustrate and compare the performances of all discussed testing procedures in a real data example
in order to confirm our conjectures concerning the advantages and disadvantages of all analyzed statistical techniques.
The considered Aalen-Johansen estimators also indicate a gender-specific influence on the CIFs for BSI between both transplant types
which may be detected with larger sample sizes.
}


\begin{figure}[ht]
        \centering
        \begin{subfigure}{0.30\textwidth}
                \includegraphics[width=\textwidth]{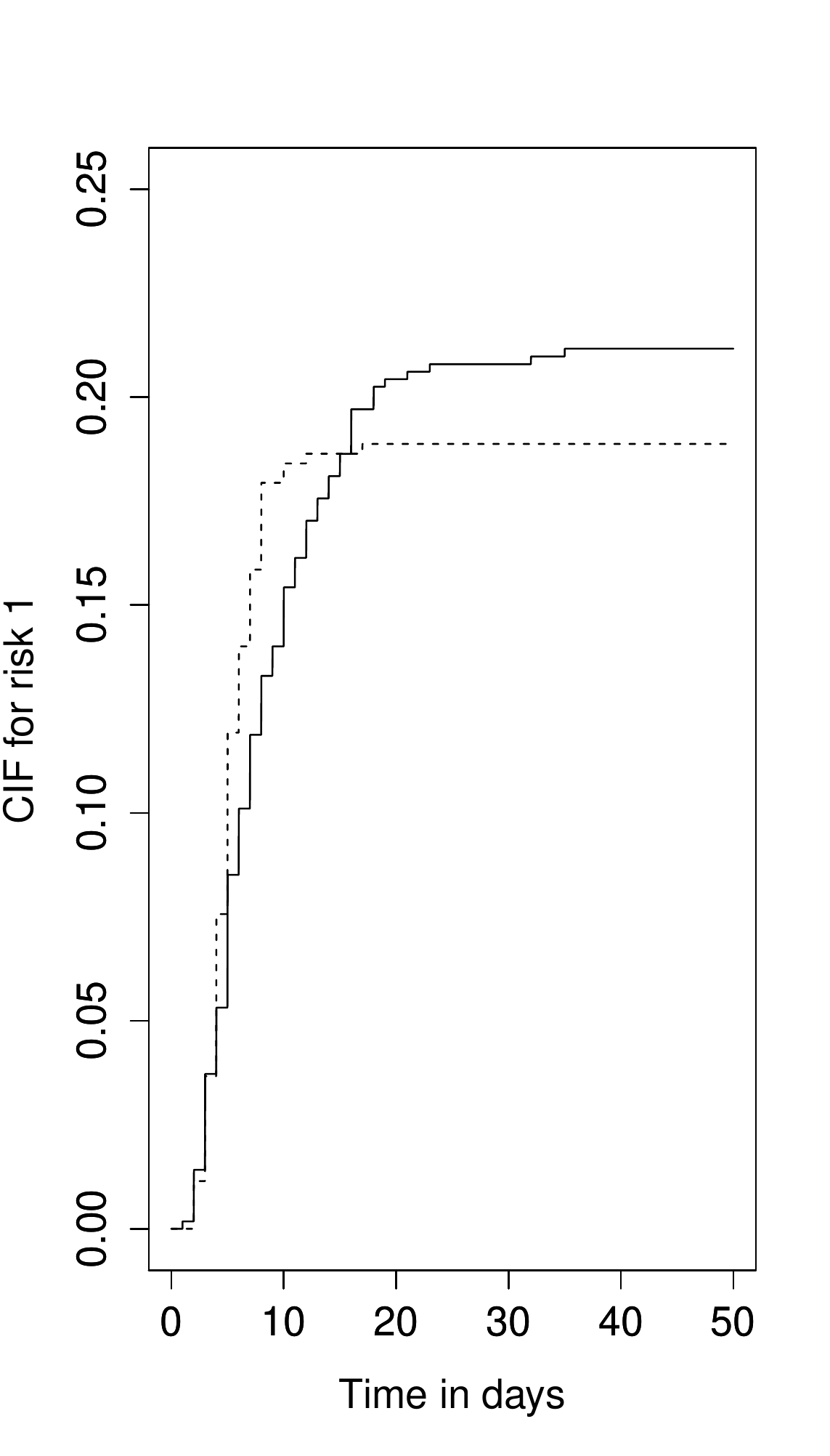}
                \caption{CIFs for BSI depending on transplant type (i):
		  allogeneic \mbox{(----)}, autologous (- - -)}
        \end{subfigure}%
        \quad 
        \begin{subfigure}{0.30\textwidth}
                \includegraphics[width=\textwidth]{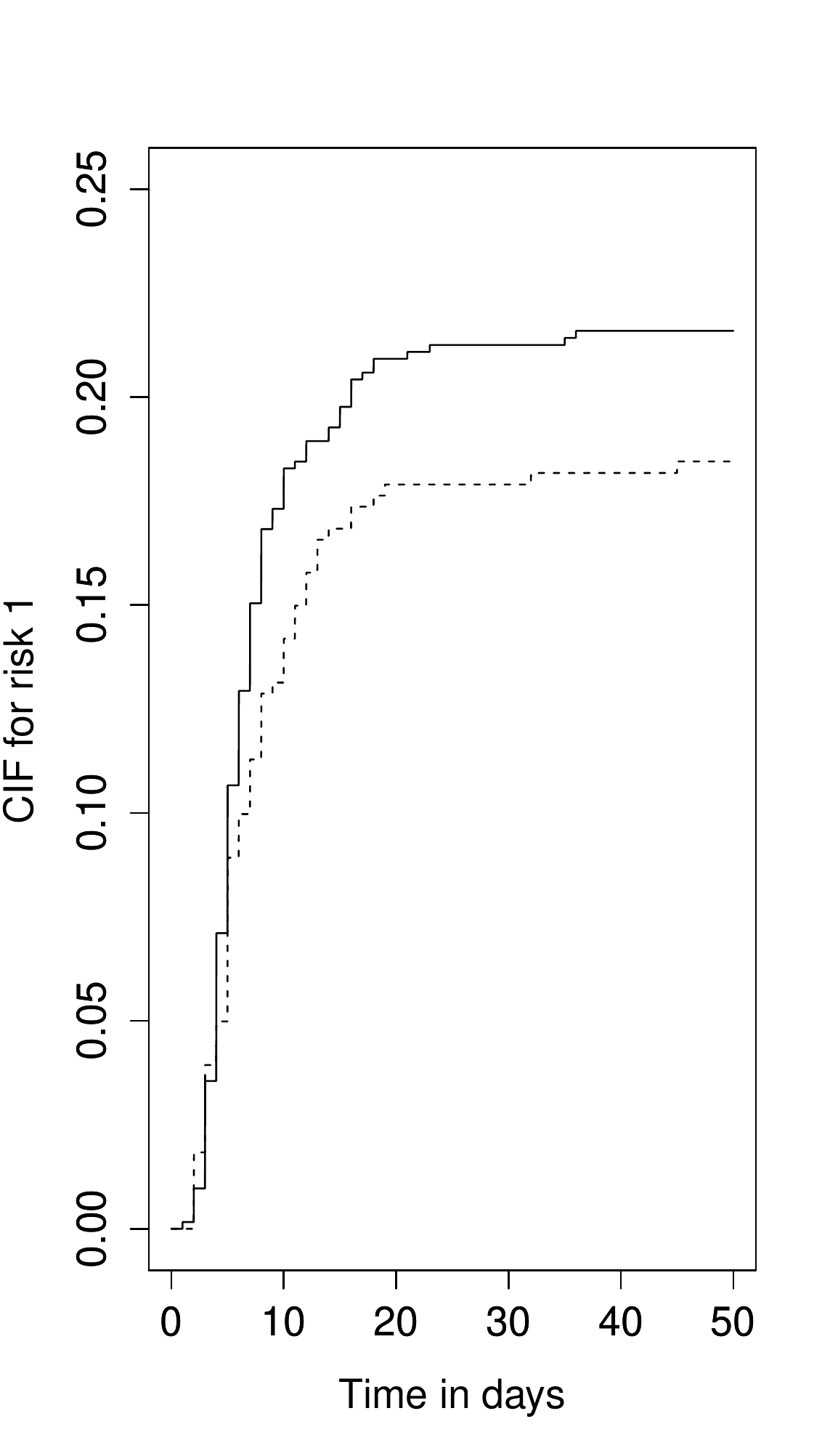}
                \caption{CIFs for BSI depending on gender (ii):
		  males (----), females (- - -)}
        \end{subfigure}
        \quad
	\begin{subtable}{0.30\textwidth}
	  \begin{tabular}{c|c|c}
	    \hline 
	    Test & (i) & (ii) \\ \hline
	    $\varphi^{KS}$ 	& .136 	& .149 \\
	    $\varphi^{CvM}$ 	& .336	& .180 \\
	    $\varphi^{B}$ 	& .314	& .155 \\
	    $\varphi^{P}$ 	& .351	& .183 \\
	    $\varphi^{Pepe}$ 	& .447	& .087 \\
	    $\varphi^{Gray}$ 	& .519	& .210 \\
	    \end{tabular}
	    \caption{$p$ values of six different tests for equal CIFs of BSI
	    between transplant types (i) and genders (ii)}
	\end{subtable}
        \caption{Plots of CIFs for BSI and $p$ values for null hypothesis $H_=$ of equal CIFs.}
	\label{fig:data_ex1}
\end{figure}

\begin{figure}[ht]
        \centering
        \begin{subfigure}{0.30\textwidth}
                \includegraphics[width=\textwidth]{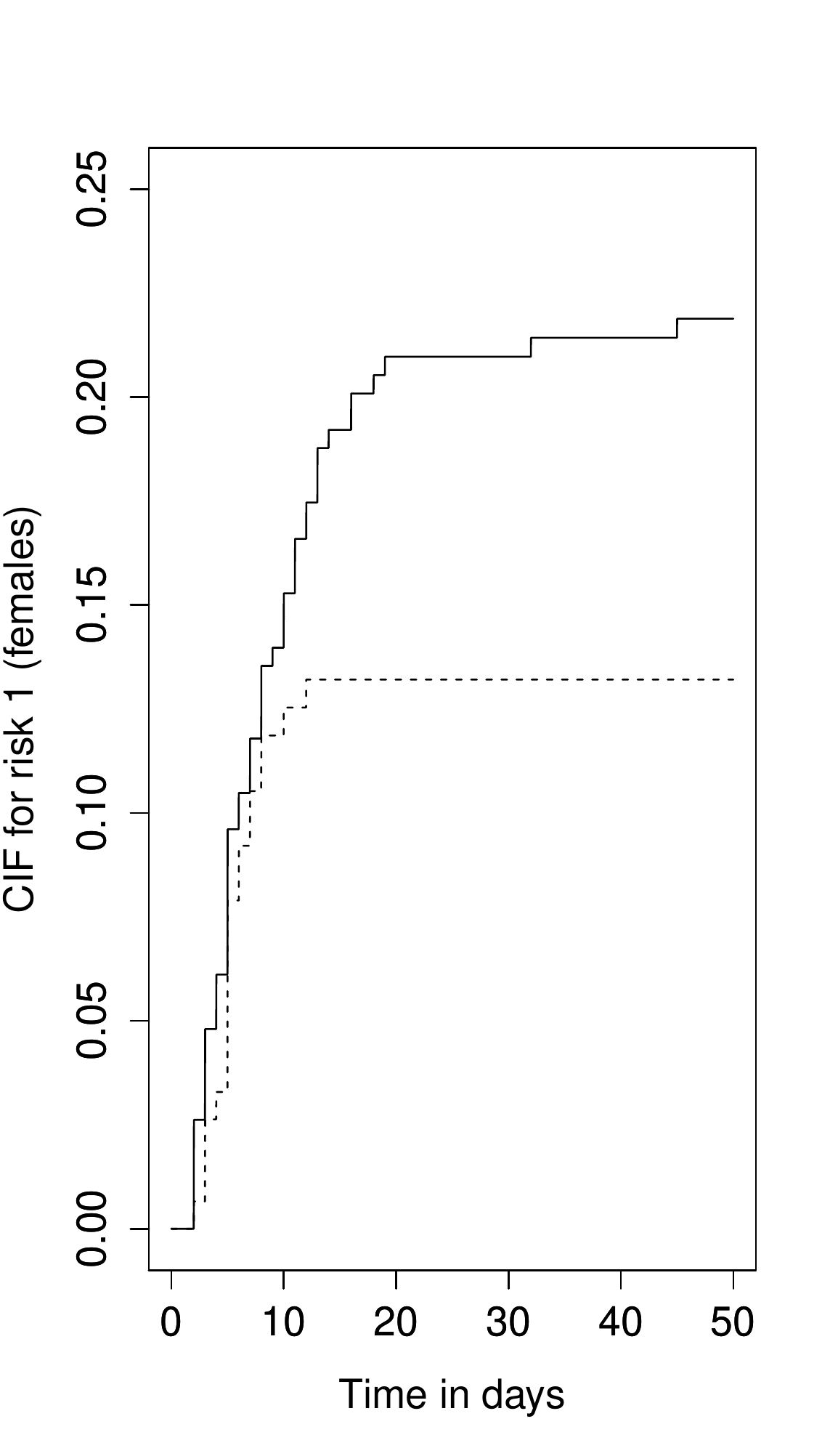}
                \caption{CIFs for BSI in females (iii)}
        \end{subfigure}%
        \quad 
        \begin{subfigure}{0.30\textwidth}
                \includegraphics[width=\textwidth]{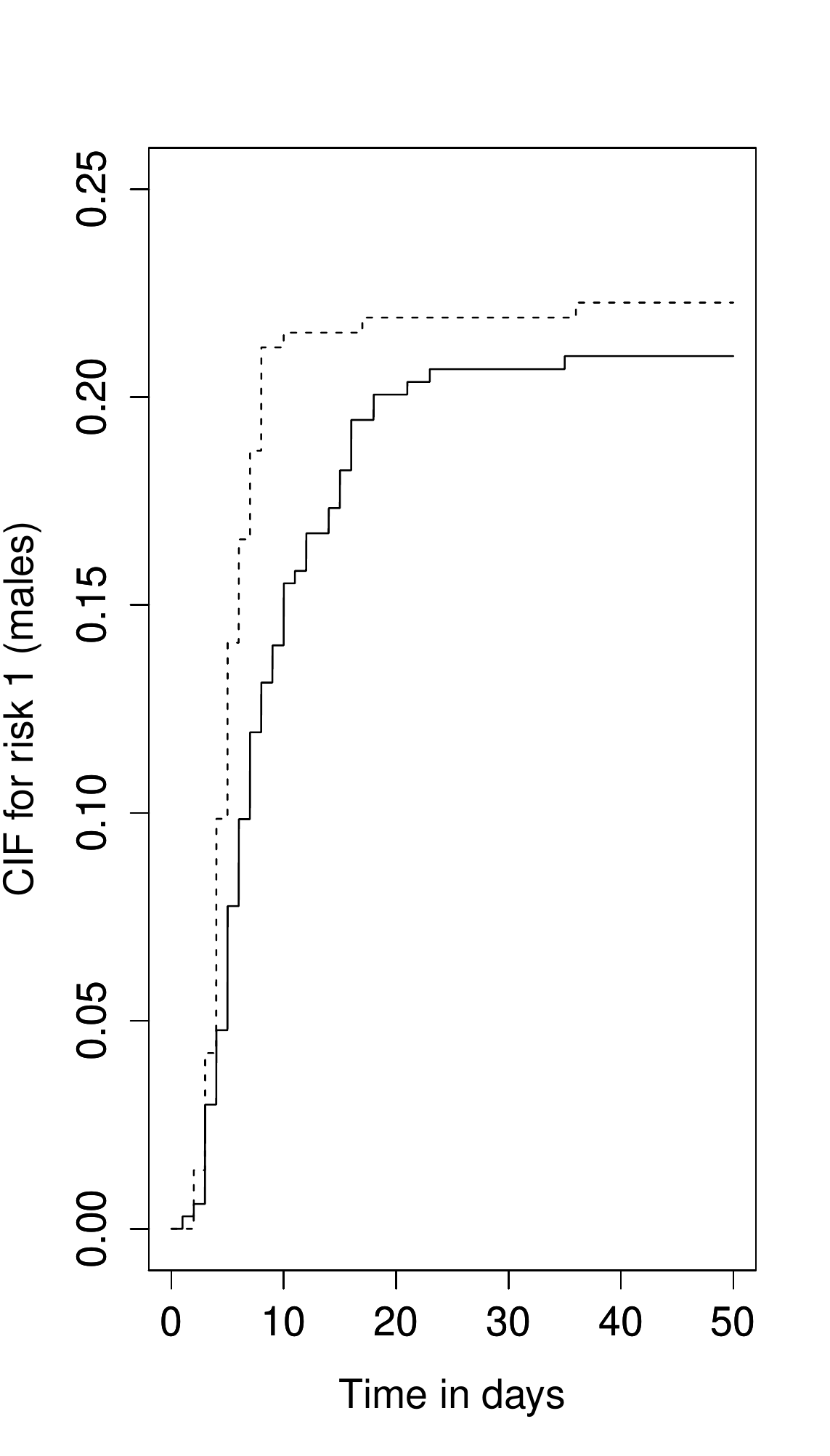}
                \caption{CIFs for BSI in males (iv)}
        \end{subfigure}
        \quad
	\begin{subtable}{0.30\textwidth}
	  \begin{tabular}{c|c|c}
	    \hline 
	    Test & (iii) & (iv) \\ \hline
	    $\varphi^{KS}$ 	& .073 & .019 \\
	    $\varphi^{CvM}$ 	& .069 & .220 \\
	    $\varphi^{B}$ 	& .058 & .193 \\
	    $\varphi^{P}$ 	& .071 & .220 \\
	    $\varphi^{Pepe}$ 	& .046 & .856 \\
	    $\varphi^{Gray}$ 	& .053 & .587 \\
	    \end{tabular}
	    \caption{$p$ values of six different tests for equal CIFs of BSI
	    between transplant types for both genders female (iii) and male (iV) separately}
	\end{subtable}
        \caption{Plots of CIFs for BSI between allogeneic (----) and autologous (- - -) transplants and $p$ values for null hypothesis $H_=$ of equal CIFs
	  after categorization of all individuals according to gender.}
	\label{fig:data_ex2}
\end{figure}

%
\section{Conclusion and Discussion}\label{sec: dis}
%
We have considered the two-sample testing problem of equality of two CIFs from two independent groups.
By only assuming the multiplicative intensity model we thereby have not only covered right-censored observations but also other situations of incomplete data as independent left-truncation.
Moreover, we have discussed and compared different test statistics based on the AJEs of the two groups. 
In particular, we have compared the Kolmogorov-Smirnov-type wild bootstrap 
test proposed in Bajorunaite and Klein (2007) with different Cram\'er-von Mises-type tests based on the wild bootstrap or different approximation techniques. 

Here the latter has not been investigated in the survival literature yet. 
All of these four tests possess asymptotic power 1, where the wild bootstrap-based versions are even asymptotically exact under the null. 
Simulations for all tests under study indicate that there is a slight but no strong preference for the wild bootstrap-based Cram\'er-von Mises test $\varphi^{CvM}$ 
for all sample sizes under consideration. In comparison the approximative Cram\'er-von Mises tests have shown an almost equally good behaviour. 
In contrast, the wild bootstrap Kolmogorov-Smirnov-type test $\varphi^{KS}$ did not seem to keep the level $\alpha$ very well in the considered set-ups.

As a concluding remark, we like to remind the reader of the advantages and disadvantages of the proposed tests.
The most important fact is the asymptotic validity of $\varphi^{KS}$ and $\varphi^{CvM}$
whereas the approximative tests $\varphi^P$ and $\varphi^B$ are no asymptotic level $\alpha$ tests.
That is, one of the first two (wild bootstrap) tests should be used whenever a large record of observations is given.
However, the sample sizes $n_1 = n_2 = 200$ are not large enough to see this difference in the present set-up.
On the other hand, $\varphi^P$ and $\varphi^B$ are more efficiently to compute by far since they do not need an additional 
Monte-Carlo step to calculate critical values. 
{\color{black}It shall be noted, that we plan to provide an \verb=R=-package containing the forementioned procedures within a larger cooperation.}


\section*{Acknowledgements}

The authors like to thank Artur Allignol, Jan Beyersmann, Tobias Bluhmki and Edgar Brunner for helpful discussions. 
Moreover, both authors would like to thank the support received by the SFF grant F-2012/375-12.

%
\section{Appendix}
%

We start to state an auxiliary result for the uniform convergence of $\hat \zeta_{n_1,n_2}$ of~\eqref{eq:zeta_hat} in probability.
This fact will be exploited to construct consistent estimators for the parameters $f, g$ and $\kappa$ from the Box and Pearson approximative tests.
\begin{lemma}
  \label{lemma:uniform_conv2}
 Let $X_n, n \geq 0,$ be a sequence of random elements 
 in the Skorohod space $\mac{D} ([0,\tau]^2)$
 and let $X_0$ be continuous and non-random.
 If, for all arguments, all $X_n$ almost surely have
the same monotonic behaviour (i.e., monotonically increasing or decreasing) and if we have convergence in probability $X_n(t) \oPo X_0(t)$
 for all $t$ in a dense subset $E^2 \subseteq [0,\tau]^2$,
 then uniform convergence in probability follows:
 \begin{align*}
  \sup_{t \in [0,\tau]^2} |X_n(t) - X_0(t)| \oPo 0
 \end{align*}
 The case with an arbitrary, finite number of arguments can be dealt with similarly.
 \end{lemma}
\noindent {\it Proof.} 
  Without loss of generality let the processes $X_n$ be non-decreasing in all arguments. For each $\varepsilon > 0$ 
  we divide $[0,\tau]^2$ into rectangles with edges $(t_j^{(1)}, t_k^{(2)}) \in E^2, j,k = 1, \dots, m,$
  where $0 = t_1^{(\ell)} < t_2^{(\ell)} < \dots < t_m^{(\ell)} = \tau, \ell = 1,2,$
  such that 
  \begin{align*}
   | X_0(t_j^{(1)}, t_k^{(2)}) - X_0(t_{j-1}^{(1)}, t_k^{(2)}) |
    \vee | X_0(t_{k}^{(1)}, t_j^{(2)}) - X_0(t_{k}^{(1)}, t_{j-1}^{(2)}) | \leq \frac{\varepsilon}{6}
  \end{align*}
  holds
  for all $2 \leq j \leq m, 1 \leq k \leq m$.
  By the subsequence principle, let $(n') \subseteq \mb{N}$ be an arbitrary subsequence and
  choose a common subsequence $(n'') \subseteq \mb{N}$ such that
  the following inequalities are almost surely true for all members of the subsequence and for all $j,k$: 
  \begin{align*}
    |X_{n''}(t_j^{(1)}, t_k^{(2)}) - X_0(t_j^{(1)}, t_k^{(2)})| < \frac{\varepsilon}{6}.
  \end{align*}
  Then, the postulated monotonicity and another application of the subsequence principle
  yield the asserted convergence: Let $t=(t^{(1)},t^{(2)}) \in [0,\tau]^2$ and fix $j,k$ giving
  $t_{j-1}^{(1)} \leq t^{(1)} \leq t_j^{(1)}$ and $t_{k-1}^{(2)} \leq t^{(2)} \leq t_k^{(2)}$, then
  \begin{align*}
    & |X_{n''}(t) - X_0(t)| \leq |X_{n''}(t_j^{(1)},t_k^{(2)}) - X_{0}(t_{j-1}^{(1)},t_{k-1}^{(2)})| \\
    & \qquad  + |X_{n''}(t_{j-1}^{(1)},t_{k-1}^{(2)}) - X_0(t_j^{(1)},t_k^{(2)})| \\
    & \quad \leq |X_{n''}(t_j^{(1)},t_k^{(2)}) - X_{0}(t_j^{(1)},t_k^{(2)})|
      + |X_{n''}(t_{j-1}^{(1)},t_{k-1}^{(2)}) - X_{0}(t_{j-1}^{(1)},t_{k-1}^{(2)})| \\
    & \qquad  + 2 |X_{0}(t_j^{(1)},t_k^{(2)}) - X_{0}(t_{j-1}^{(1)},t_{k-1}^{(2)})|
      \leq \frac{\varepsilon}{6} + \frac{\varepsilon}{6} + 4 \frac{\varepsilon}{6} = \varepsilon.
  \end{align*}
  \hfill $\Box$

\begin{cor}
\label{cor:unif_conv_zeta_hat}
  Let $t < \tau$,
  then $\hat \zeta_{n_1,n_2}$ from~\eqref{eq:zeta_hat} converges uniformly on $[0,t]^2$
  to the covariance function~\eqref{eq:zeta_V} of the Gaussian process $V$  
  in probability, as $n \rightarrow \infty$ and $\frac{n_1}{n} \rightarrow p \in (0,1)$.
\end{cor}
\begin{proof}
  It suffices to prove consistency of $\zeta_{n_k}^{(k)}, k=1,2,$ defined in \eqref{eq: zeta_hat_k}.
Due to similarity, we focus on the first integral which can be decomposed as
  \begin{align*}
    & n_k \int_0^{s_1 \wedge s_2} \frac{ \{ \hat S_2^{(k)}(u) - \hat F_1^{(k)}(s_1) \} 
      \{ \hat S_2^{(k)}(u) - \hat F_1^{(k)}(s_2) \} } {(Y^{(k)})^2(u)} \d N_1^{(k)}(u) \\
    &  = n_k \int_0^{s_1 \wedge s_2} \frac{ (\hat S_2^{(k)})^2 } {(Y^{(k)})^2} \d N_1^{(k)}
     - (\hat F_1^{(k)}(s_1) + \hat F_1^{(k)}(s_2)) n_k \int_0^{s_1 \wedge s_2} \frac{ \hat S_2^{(k)} }
     {(Y^{(k)})^2} \d N_1^{(k)} \\
    & \qquad + \hat F_1^{(k)}(s_1) \hat F_1^{(k)}(s_2) n_k \int_0^{s_1 \wedge s_2} \frac{ \d N_1^{(k)} } {(Y^{(k)})^2}.
  \end{align*}
  The CIFs in the above expression converge uniformly in probability, see Andersen et al. (1993). 
  With arguments similar to those presented in Beyersmann et al. (2013) for the convergence of the covariance estimator 
  in probability, it can be shown that, for all fixed $r,s$, all of the above integrals 
  converge in probability to their real counterparts
  \begin{align*}
   \int_0^{r \wedge s} \frac{ (S_2^{(k)})^h(u) \alpha_{1}^{(k)}(u) } {y^{(k)}(u)} \d u, \; h=0,1,2.
  \end{align*}
  Thus, an application of Lemma~\ref{lemma:uniform_conv2} concludes this proof.
\end{proof}
\noindent {\it Proof of Theorem~\ref{thm:test statistics}.}
The stated convergences of both test statistics are direct consequences of the continuous mapping theorem and Theorem~\ref{thm:sqrt_n_conv_cifs}. 
Moreover, the representation of $T^{CvM}$ as a weighted sum of $\chi^2$-distributed random variables is a consequence of Mercer's Theorem; see e.g., Theorem~3.15 in Adler (1990). \nocite{adler90}
However, for sake of completeness we shortly outline its proof. 
Note first, that by turning to $\rho_2^{1/2} V$ instead of $V$ we can without loss of generality assume that $\rho_2 \equiv 1$ holds since $\rho_2$ is continuous. 
Now denote all (normalized) eigenfunctions and eigenvalues of the integral equation
    \begin{align} \label{eq:eigen}
    \int_{t_1}^{t_2} \zeta(u,s) e(s) \d s = \lambda e(u) \quad \text{for all} \quad u \in [t_1,t_2]
    \end{align}
    by $(e_j)_j$ and $(\lambda_j)_j$, respectively.
    That is, $\int_{t_1}^{t_2} e_{i}(s) e_{j}(s) \d s = \delta_{ij}$, where $\delta_{ij} = \mathbf{1}\{i=j\}$ denotes Kronecker's delta.
    Mercer's Theorem then implies that the covariance function $\zeta_V$ admits a decomposition as
    \begin{align} \begin{split} \label{eq:mercer}
      \zeta_V(s_1,s_2) = \sum_{j=1}^\infty \lambda_j e_j(s_1) e_j(s_2),
    \end{split} \end{align}
    where the convergence is absolute and uniform on $[t_1,t_2]^2$.
    Now the Karhunen-Lo\`{e}ve Theorem (by combining Theorems~3.7 and~3.16 in Adler, 1990) states that $V$ admits the expansion
    \begin{align}
    \label{eq:karhunen_loeve}
    V(s) = \sum_{j=1}^\infty \lambda_j^{1/2} Z_j e_j(s)
    \end{align}
    where the $Z_j$ are i.i.d. standard normally distributed
    and the equality is understood to be equality in law. 
    Due to the finiteness of all integrals and sums ($\sum_{j=1}^\infty \lambda_j = \int \zeta_V(s,s) \d s  <\infty$  by monotone convergence), 
    we can change the order of integration in $\int_{t_1}^{t_2} V^2(u) \d u $ with the help of Fubini's theorem,
    use the orthonormality of $(e_j)_j$ and arrive at the desired representation.
\hfill$\Box$
\begin{proof}[Proof of Theorem~\ref{thm:box_approx}]
It is sufficient to prove consistency of $\hat \mu_{n_1,n_2}$ and $\hat \sigma_{n_1,n_2}^2$ 
    for $\mu$ and $\sigma^2$, respectively. 
The consistency of $\hat \mu_{n_1,n_2}$ for $\int_{t_1}^{t_2} \zeta_V(s,s) \d s  = \sum_{j=1}^\infty \lambda_j =\mu$ follows directly from the uniform convergence of $\hat \zeta_{n_1,n_2}$ in probability stated in Corollary~\ref{cor:unif_conv_zeta_hat}. 
For $\hat \sigma_{n_1,n_2}^2$, remark that the Decomposition~\eqref{eq:mercer}, Fubini's Theorem, the orthonormality of $(e_j)_j$ and the dominated convergence theorem yield
    \begin{align*}
    Var(Q) & = Var \Big( \sum_{j=1}^\infty \lambda_j Z_j^2 \Big) 
    = 2\sum_{j=1}^\infty \lambda_j^2 \\
    & = 2\sum_{i,j} \lambda_i \lambda_j \left( \int_{t_1}^{t_2} e_i(s) e_j(s) \d s \right)^2 \\
    & = 2\int_{[t_1,t_2]^2} \zeta_V^2(s_1,s_2) \d \leb^2 (s_1,s_2), 
    \end{align*}
    where the applicability of the theorems is justified by the following bound (obtained from Cauchy-Schwarz and monotone convergence)
    \begin{align*}
    \int_{[t_1,t_2]^2} \zeta_V^2(s_1,s_2) \d \leb^2 (s_1,s_2)
    & \leq \int_{[t_1,t_2]^2} \Big( \sum_{j=1}^\infty \lambda_j |e_j(s_1) e_j(s_2)| \Big)^2 \d \leb^2 (s_1,s_2) \\
    & \leq \int_{[t_1,t_2]^2} \Big( \sum_{j=1}^\infty \lambda_j e_j^2(s_1) \Big)
	\Big( \sum_{j=1}^\infty \lambda_j e_j^2(s_2) \Big) \d \leb^2 (s_1,s_2) \\
    & = \Big( \sum_{j=1}^\infty \lambda_j \Big)^2 < \infty.
    \end{align*}
As for $\hat \mu_{n_1,n_2}$, the consistency of $\hat \sigma_{n_1,n_2}^2$ for $2 \int_{[t_1,t_2]^2} \zeta_V^2(s_1,s_2) \d \leb^2 (s_1,s_2) = 2 \sum_{j=1}^\infty \lambda_j^2 =\sigma^2$ follows which completes the proof.
\end{proof}

\begin{proof}[Proof of Theorem~\ref{thm:pearson_approx}]
As above we may assume $\rho_2 \equiv 1$ without loss of generality.
Recall that the skewness of $\chi_\kappa^2$, i.e., a $\Gamma(\kappa/2,2)$-gamma distribution, is given by $\sqrt{\frac{8}{\kappa}}$. 
Moreover, it follows from the independence of $Z_i$ and $Z_j$, $i \neq j$, that the skewness of $Q_{\textnormal{stud}}$ equals $\sigma^{-3}$ times
    \begin{align*}
      \E[(Q - \E[Q])^3] &
      = \E \Big[ \Big(\sum_{j=1}^\infty \lambda_j (Z_j^2 - 1) \Big)^3 \Big] \\
      & = \sum_{i,j,k} \lambda_i \lambda_j \lambda_k \E[(Z_i^2 - 1)(Z_j^2 - 1)(Z_k^2 - 1)] \\
      & = \sum_{j=1}^\infty \lambda_j^3 \E[(Z_j^2 - 1)^3] = 8 \sum_{j=1}^\infty \lambda_j^3 .
    \end{align*}
Divided by 8 this equals $\sum_{i,j,k} \lambda_i \lambda_j \lambda_k \delta_{ik} \delta_{ij} \delta_{jk}$
which can be rewritten by Mercer's Theorem as
\begin{align*}
 & \sum_{i,j,k} \lambda_i \lambda_j \lambda_k \int_{t_1}^{t_2} e_i(s_1) e_k(s_1) \d s_1 \int_{t_1}^{t_2} e_i(s_2) e_j(s_2) \d s_2 \int_{t_1}^{t_2} e_j(s_3) e_k(s_3) \d s_3 \\
 & = \int_{[t_1,t_2]^3} \sum_{i=1}^\infty \lambda_i e_i(s_1) e_i(s_2) \sum_{j=1}^\infty \lambda_j e_j(s_2) e_j(s_3) \sum_{k=1}^\infty \lambda_k e_k(s_3) e_k(s_1) \d \leb^3(s_1,s_2,s_3) \\
 & = \int_{[t_1,t_2]^3} \zeta_V(s_1,s_2) \zeta_V(s_2,s_3) \zeta_V(s_3,s_1) \d \leb^3(s_1,s_2,s_3);
\end{align*}
see also the monograph of Shorack and Wellner (2009), \nocite{shorack09} the equation following~5.2.(20) therein.
The justification for the exchangeability of the above sums and integrals is given in the same manner as in the previous proof.
Equating these quantities it follows that $\kappa$ should equal $(\sum_{j=1}^\infty \lambda_j^2)^3/(\sum_{j=1}^\infty \lambda_j^3)^{2}$. In particular, this choice also guarantees equality of the first two moments of $Q_{\textnormal{stud}}$ and $\chi^2_{\kappa,\textnormal{stud}}$. 
Now, as proven in Theorem~\ref{thm:box_approx}, $\frac12 \hat \sigma_{n_1,n_2}^2$ is a consistent estimator for $\sum_{j=1}^\infty \lambda_j^2$. Moreover by Corollary~\ref{cor:unif_conv_zeta_hat}, $\hat \gamma_{n_1,n_2}$
is consistent for $\sum_{j=1}^\infty \lambda_j^3$.
All in all, this shows that $\hat \kappa$ is consistent for $\kappa$.
\end{proof}

\section{Supplement to the Simulation Section~\ref{sec: sim}}
\label{sec: appA}

{\color{black}Supplementing the first two simulation scenarios we here present a third one which has been adopted from Dobler and Pauly (2014)\nocite{dobler14}:}
\begin{enumerate}
 \item[3.] The event times are given by the cause-specific hazard intensities
  \begin{align*}
    \alpha_{1}^{(1)}(u) = \exp(-u), \quad \alpha_{2}^{(1)}(u) = 1-\exp(-u) \quad \text{and} \quad
    \alpha_{1}^{(2)} \equiv c \equiv 2-\alpha_{2}^{(2)}, 
  \end{align*}
  where $0\leq c\leq 1$.
  The case $c=1$ is equivalent to the presence of the null hypothesis $H_=$,
  whereas both CIFs for the first competing risk are located deeper 
  in the alternative hypothesis $H_{\neq}$ as $c<1$ decreases. 
  The examined sample sizes are the same as {\color{black}in the first two simulation scenarios given in the paper}
  and the domain of interest equals $[t_1,t_2] = [0,1.5]$.
  The data are independently right-censored by different exponential $Exp(\lambda^{(k)})$-distributions
  with pdfs $f^{(k)}(x) = \lambda^{(k)} \exp(-\lambda^{(k)} x) {\bf 1}_{(0,\infty)}(x)$.
  $\lambda^{(k)} = 0.5, 1$ corresponds to light and moderate censoring, respectively,
  and $\lambda^{(k)}=0$ indicates the uncensored case.
\end{enumerate}

\iftrue

The simulated effective type-I error probabilities of the resampling tests $\varphi^{KS},\varphi^{CvM},\varphi^{Pepe}$
as well as those of the approximative tests $\varphi^P$ and $\varphi^B$ in {\color{black}this} third set-up
can be found in Table~\ref{table:niveau_DP}.
Since the Kolmogorov-Smirnov test is the most liberal one by far (except the case $n_1=n_2=200$)
and it therefore obviously has the greatest power (cf. Tables~\ref{table:power_BK_1} and~\ref{table:power_BK_2}), 
its results from the simulations for assessing the power behaviour are not presented in Table~\ref{table:power_DP} {\color{black}below}.
The remaining tests wrongly reject the null hypothesis $H_=$ with more acceptable rates
-- in fact, the sizes of the tests based on the Cram\'er-von Mises statistic 
do not differ very much among one another, which already has been seen in Table~\ref{table:niveau_BK_1} in a different set-up.
On the one hand, all three tests $\varphi^{CvM}$, $\varphi^P$ and $\varphi^B$ are slightly too liberal 
when censoring or considerably unequal sample sizes are present.
This observation contradicts our expectation 
that the approximative tests are constructed by means of conservative critical values.
On the other hand, however, the prescribed level $\alpha = 0.05$ is maintained excellently
for uncensored and equally sized sample groups even for small sample sizes such as $n_1 = n_2 = 50$.
The Pepe-type test keeps the nominal level best by far, it can even handle extremely small samples as well as moderate censoring.

Let us now consider the simulated power of $\varphi^{CvM},\varphi^P,\varphi^B$ and $\varphi^{Pepe}$.
Therefore, we have chosen the CIFs of the second group corresponding to the parameters $c=0.9, 0.8, \dots, 0.1$
and we have only considered the cases where $n_1 = n_2 \in \{50, 100\}$ and $\lambda^{(1)} = \lambda^{(2)} \in \{0,1\}$.
As usual the power increases as the distance to the null hypothesis grows. 
Further, it strikes the eye that both approximative tests $\varphi^P$ and $\varphi^B$ share the same power in most cases under consideration. 
Since they also keep the level $\alpha = 0.05$ nearly equally well,
there is no clear preference for one of both tests. 
When compared to the wild bootstrap test,
we see that $\varphi^{CvM}$ in many cases has the highest power (differences up to $.01$)
whereas in some cases the approximative tests are superior (differences up to $.004$).
{\color{black}To sum up,} all three tests show a comparable behaviour under $H_=$,
The Pepe-type test $\varphi^{Pepe}$ not only keeps the prescribed level excellently,
it also has the highest power of all presented tests.
{\color{black}However, this behaviour is dearly bought with the lack of power in situations such as the first two simulation scenarios reflect;
see Remark~\ref{rem:test statistics} as well as the discussion in Section~\ref{sec: sim}.}

\begin{sidewaystable}
\begin{center}
\setlength{\tabcolsep}{5pt}
\begin{tabular}[h]{cc|c|c|c|c|c|c|c|c|c|c|c|c|c|c|c}
\hline
 & \multicolumn{1}{r|}{$(n_1,n_2)$} & \multicolumn{5}{|c|}{(20,20)} & \multicolumn{5}{|c|}{(50,50)} & \multicolumn{5}{|c}{(50,100)} \\
\multicolumn{2}{c|}{$(\lambda^{(1)},\lambda^{(2)})$}  & $\varphi^{KS}$ & $\varphi^{CvM}$ & $\varphi^{P}$ & $\varphi^{B}$ & $\varphi^{Pepe}$ & $\varphi^{KS}$ & $\varphi^{CvM}$ & $\varphi^{P}$ & $\varphi^{B}$ & $\varphi^{Pepe}$ & $\varphi^{KS}$ & $\varphi^{CvM}$ & $\varphi^{P}$ & $\varphi^{B}$ & $\varphi^{Pepe}$ \\\hline
\multicolumn{2}{c|}{(0,0)}  & .091 & .078 & .080 & .080 & \bf .064 & .092 & \bf .051 & .048 & .048 & .064 & \bf .053 & \bf .053 & .054 & .054 & \bf .053\\
\multicolumn{2}{c|}{(0.5,1)} & .130 & .094 & .086 & .087 & \bf .067 & .080 & .068 & .071 & .071 & \bf .047 & .074 & .057 & \bf .053 & \bf .053 & .066 \\
\multicolumn{2}{c|}{(1,0.5)} & .128 & .095 & .088 & .089 & \bf .077 & .079 & .061 & .061 & .061 & \bf .046 & .080 & .068 & .066 & .067 & \bf .055 \\
\multicolumn{2}{c|}{(1,1)} & .135 & .107 & .107 & .107 & \bf .078 & .090 & .077 & .073 & .073 & \bf .054 & .073 & .067 & .066 & .066 & \bf .049 \\
\hline
 & \multicolumn{1}{r|}{$(n_1,n_2)$} & \multicolumn{5}{|c|}{(100,50)} & \multicolumn{5}{|c|}{(100,100)} & \multicolumn{5}{|c}{(200,200)} \\
\multicolumn{2}{c|}{$(\lambda^{(1)},\lambda^{(2)})$}  & $\varphi^{KS}$ & $\varphi^{CvM}$ & $\varphi^{P}$ & $\varphi^{B}$ & $\varphi^{Pepe}$ & $\varphi^{KS}$ & $\varphi^{CvM}$ & $\varphi^{P}$ & $\varphi^{B}$ & $\varphi^{Pepe}$ & $\varphi^{KS}$ & $\varphi^{CvM}$ & $\varphi^{P}$ & $\varphi^{B}$ & $\varphi^{Pepe}$ \\\hline
\multicolumn{2}{c|}{(0,0)} & .073 & \bf .057 & .058 & .058 & .059 & .059 & \bf .051 & \bf .051 & \bf .051 & .058 & .044 & .057 & \bf .051 & \bf .051 & .053 \\
\multicolumn{2}{c|}{(0.5,1)} & .108 & .057 & \bf .053 & \bf .053 & .074 & .063 & .074 & .072 & .072 & \bf .055 & .081 & .063 & \bf .061 & \bf .061 & .064 \\
\multicolumn{2}{c|}{(1,0.5)} &.067 & .056 & .056 & .056 & \bf .052 & .068 & .064 & .068 & .069 & \bf .061 & .057 & \bf .055 & .058 & .059 & .039 \\
\multicolumn{2}{c|}{(1,1)} &.099 & \bf .063 & .064 & .064 & .072 & .081 & .070 & \bf .067 & \bf .067 & \bf .067 & .052 & .066 & .063 & .063 & \bf .051
\end{tabular}
\caption{(Model of Dobler and Pauly (2014)) Simulated sizes of the resampling tests $\varphi^{KS}, \varphi^{CvM}, \varphi^{Pepe}$
  and the approximative tests $\varphi^P, \varphi^B$ for nominal size $\alpha=5\%$
  under different sample sizes and censoring distributions under $H_=$.}
\label{table:niveau_DP}
\end{center}
\end{sidewaystable}


\begin{sidewaystable}
\begin{center}
\setlength{\tabcolsep}{5pt}
\begin{tabular}[h]{c|c|c|c|c|c|c|c|c|c|c|c|c|c|c|c|c}
\hline
 \multicolumn{1}{r|}{$(n_1,n_2)$} & \multicolumn{8}{|c|}{(50,50)} & \multicolumn{8}{|c}{(100,100)}  \\
 \multicolumn{1}{r|}{$(\lambda^{(1)},\lambda^{(2)})$}  & \multicolumn{4}{|c|}{(0,0)} & \multicolumn{4}{|c|}{(1,1)} & \multicolumn{4}{|c|}{(0,0)} & \multicolumn{4}{|c}{(1,1)}\\
c  & $\varphi^{CvM}$ & $\varphi^{P}$ & $\varphi^{B}$ & $\varphi^{Pepe}$ & $\varphi^{CvM}$ & $\varphi^{P}$ & $\varphi^{B}$ & $\varphi^{Pepe}$ & $\varphi^{CvM}$ & $\varphi^{P}$ & $\varphi^{B}$ & $\varphi^{Pepe}$ & $\varphi^{CvM}$ & $\varphi^{P}$ & $\varphi^{B}$ & $\varphi^{Pepe}$ \\\hline
0.9 & .083 & .085 & .085 & .111 & .080 & .080 & .080 & .110 & .093 & .096 & .096 & .162 & .104 & .103 & .103 & .152 \\
0.8 & .166 & .160 & .160 & .249 & .140 & .140 & .140 & .214 & .239 & .239 & .239 & .352 & .206 & .210 & .210 & .294 \\
0.7 & .305 & .297 & .297 & .392 & .228 & .229 & .229 & .352 & .490 & .485 & .485 & .597 & .387 & .382 & .382 & .520 \\
0.6 & .492 & .485 & .485 & .573 & .388 & .391 & .391 & .495 & .772 & .773 & .773 & .824 & .625 & .623 & .623 & .719 \\
0.5 & .674 & .671 & .671 & .761 & .541 & .538 & .538 & .663 & .926 & .928 & .928 & .946 & .814 & .808 & .808 & .891 \\
0.4 & .840 & .844 & .842 & .892 & .707 & .704 & .704 & .834 & .981 & .981 & .981 & .997 & .934 & .933 & .933 & .967 \\
0.3 & .949 & .949 & .949 & .978 & .871 & .861 & .861 & .921 & .999 & .999 & .999 & 1 & .991 & .989 & .989 & .998 \\
0.2 & .989 & .989 & .989 & .994 & .949 & .950 & .950 & .970 & 1 & 1 & 1 & 1 & .999 & .999 & .999 & 1 \\
0.1 & 1 & 1 & 1 & 1 & .993 & .994 & .994 & .998 & 1 & 1 & 1 & 1 & 1 & 1 & 1 & 1 \\
\end{tabular}
\caption{(Model of Dobler and Pauly (2014)) Simulated power of the resampling tests $\varphi^{CvM}, \varphi^{Pepe}$
  and the approximative tests $\varphi^P, \varphi^B$ for nominal size $\alpha=5\%$ 
  under different sample sizes and censoring distributions under $H_{\neq}$.}
\label{table:power_DP}
\end{center}
\end{sidewaystable}
\fi

%

%
%

\bibliography{literatur}
\bibliographystyle{plain}

\end{document}